\documentclass[11pt]{amsart}
\usepackage{graphicx, amssymb, amsmath, amsthm}
\usepackage{epsfig}
\usepackage{hyperref}
\usepackage{mathrsfs}

\usepackage{tikz}
\usepackage{here}

\usepackage{pifont}

\usepackage{listings}
\numberwithin{equation}{section}
\usepackage{xcolor}

\newtheorem{theorem}{Theorem}[section]
\newtheorem{prop}[theorem]{Proposition}
\newtheorem{lemma}[theorem]{Lemma}

\newtheorem{corollary}[theorem]{Corollary}

\newtheorem{proposition}[theorem]{Proposition}
\newtheorem{example}[theorem]{Example}

\theoremstyle{definition}
\newtheorem{definition}[theorem]{Definition}
\newtheorem{remark}[theorem]{Remark}

\makeatletter
\newcommand{\Extend}[5]{\ext@arrow0099{\arrowfill@#1#2#3}{#4}{#5}}
\makeatother

\newcommand{\R}{\mathbb{R}}  
\newcommand{\Z}{\mathbb{Z}}  
\newcommand{\N}{\mathbb{N}}  
\newcommand{\CC}{\mathbb{C}}  
\newcommand{\II}{\mathrm{i}}
\newcommand{\p}{\partial}  
\newcommand{\pb}{\bar{\partial}}
\newcommand{\ud}{\mathrm{d}}

\newcommand{\Tr}{\mathrm{Tr}}

\newcommand{\bs}[1]{\boldsymbol{#1}}
\newcommand{\depp}[1]{\mathrm{depth}(#1)}

\begin{document}
\title[A Gluing Theorem For Collapsing WQAC-CY Manifolds]{A Gluing Theorem For Collapsing Warped-QAC Calabi-Yau Manifolds}

\author[Dashen Yan]{Dashen Yan}
\address{Department of Mathematics, Stony Brook University, Stony Brook, NY 11794-3651 USA. E-mail: dashen.yan@math.stonybrook.edu}

\maketitle
\begin{abstract}
	We carry out a gluing construction for collapsing warped-QAC (quasi-asymptotically-conical) Calabi-Yau manifolds in $\CC^{n+2}, n\geq 2$. This gluing theorem verifies a conjecture  by Yang Li in \cite{li2019gluing} on the behavior of the warped QAC Calabi-Yau metrics on affine quadrics when two singular fibers of a holomorphic fibration go apart. We will also discuss a bubble tree structure for those collapsing warped-QAC Calabi-Yau manifolds.
\end{abstract}

\section{Introduction and background}\label{S_1}
	The goal of this paper is to understand the structure of blow-up limits of collapse Calabi-Yau manifolds along holomorphic fibrations when several singular fibers collide.\\
	
	A holomorphic fibration from a compact Calabi-Yau manifold to some base space will generally contains singular fibers. In the region away from the singular fibers, the geometry of a compact Calabi-Yau manifold collapsing along holomorphic fibration has been previously studied elsewhere. For example, if the volume of the smooth fibers is normalized to a fixed constant, there is a semi-Ricci-flat description of the Calabi-Yau metric (e.g. \cite{tosatti2010adiabatic}). Indeed, in a small tubular neighborhood of $F$, this metric looks like a product of the Calabi-Yau metric on $F$ and a flat metric on the base space. By contrast, in the case of a collapsing Calabi-Yau 3-fold admitting a Lefschetz $K3$ fibration, a refined description was obtained by Yang Li in \cite{li2019new} via a gluing construction. Near the nodal singularities of the singular fibers, the metric is approximated by a scaled Calabi-Yau metric $(\CC^{3},\omega_{\CC^{3}})$ constructed in \cite{li2019new,szekelyhidi2019degenerations,conlon2017new}. The holomorphic fibration structure on $\CC^{3}$ relating to this collapsing problem is Lefschetz fibration
		\begin{equation*}
			\pi:\CC^{3}\to \CC,(z_{1},z_{2},z_{3})\mapsto z_{1}^{2}+z_{2}^{2}+z_{3}^{2}.
		\end{equation*}
	The generic fibers of $\pi$ are smoothings of $\CC^{2}/\Z_{2}$ and are mutually biholomorphic. Outside a large compact set, the Calabi-Yau metric $g_{\omega_{\CC^{3}}}$ is locally modeled on the product of a flat metric on $\CC$ and $|\pi(\vec{z})|^{1/2}g_{S}$, where $g_{S}$ is the Stenzel metric on the smooth fiber $z_{1}^{2}+z_{2}^{2}+z_{3}^{2}+1=0$.\\
	
 	Examples of complete Calabi-Yau manifolds similar to this example $(\CC^{3},\omega_{\CC^{3}})$ have been constructed in \cite{szekelyhidi2019degenerations,conlon2017new,chiu2022nonuniqueness,firester2024complete,conlon2023warped}. Those manifolds have maximal volume growth and a tangent cone at infinity with a singular cross section. Moreover, the underlying manifolds admit a holomorphic fibration to $\CC$ with generic fibers that are a smoothing of Calabi-Yau cones. 

 	 Outside a large compact set, the metrics are modeled on a warped product of the flat metric on the base space and an asymptotically conical Calabi-Yau metric on the fibers, which we use \textbf{warped QAC} (warped quasi-asymptotically conical) to denote this properties \footnote{The precise definition for warped QAC metric is given in Definition 2.9 in \cite{conlon2023warped}. We borrow the terminology in this paper.}.\\
 	
 	Such warped QAC Calabi-Yau metrics are candidates for blow-up limits of the collapsing Calabi-Yau manifold along holomorphic fibration near the critical points of the fibrations. Roughly speaking, one can think of a warped QAC metrics interpolating the geometry between the critical points, where the semi-Ricci-flat description breaks down, and the semi-Ricci-flat geometry away from the singular fibers, by resolving a conical singularity fiberwise.\\
 	
 	The starting point of this paper is the situation where, as the metric collapses, the singular fibers are allowed to collide. Depending on relative rate of colliding and collapsing, the blow-up limits fall into three categories: colliding is faster than, comparable to, or slower than collapsing. Some examples for the first two cases are considered by Conlon and Rochon in \cite{conlon2023warped}. Their approach requires the Laplacian to be an isomorphism between appropriate weighted H\"older spaces. In the first case, which generalizes the work of Hein and Sun on Calabi-Yau manifolds with isolated conical singularities \cite{hein2017calabi} to a non-compact setting, this condition of the Laplacian is needed to solve a singular Monge-Amp\`ere equation.\\
 
 	In present paper, we carry out a gluing construction for the blow-up limit in the case when colliding is slower than collapsing. At the outset, we provide a construction of warped QAC Calabi-Yau metrics, that extends a construction of Sz\'ekelyhidi's in \cite{szekelyhidi2019degenerations}. Let $X$ be a smooth affine hypersurface in $\CC^{n+2}$ defined by the following equation
 		\begin{equation}\label{S_1_Defn_X}
 			P(\tilde{z})+Q(z_{n+2})=0, (\tilde{z},z_{n+2})\in \CC^{n+2}, n\geq 2.
 		\end{equation}
 	Let $V_{0}\subset \CC^{n+1}$ defined by $P(\tilde{z})=0$, such that $V_{0}$ admits a Calabi-Yau cone metric and denote its K\"ahler form as $\omega_{0}=\II \p \pb r^{2}$. We assume that there is a weighted action on $\CC^{n+2}$ that preserves $V_{0}$
 		\begin{equation*}
 			a\cdot \tilde{z}=(a^{w_{1}}z_{1}, \cdots, a^{w_{n+1}}z_{n+1}).
 		\end{equation*}
 	The weight $w_{i}$ are normalized so that the action is a rescaling on $V_{0}$, in the sense that $r^{2}(a\cdot \tilde{z})=|a|^{2}r^{2}(\tilde{z})$. We also assume that the polynomial $P$ is homogeneous under the weighted action $P(a \cdot \tilde{z})=a^{p}P(\tilde{z}), a\in \CC^{*}$. We will call $p$ the degree of $P$, and also we will let $q$ to denote the degree of the polynomial $Q(z_{n+2})$. Now define the holomorphic fibration $\pi$ to be the projection to the $z_{n+2}$ factor:
		\begin{equation*}
			\pi: X\to \CC.
		\end{equation*}
	Our first main result is:
		\begin{theorem}\label{Thm_1}
			If the degree $p$ of $P$ is larger than the degree $q$ of $Q$, there exists a warped QAC Calabi-Yau K\"ahler form $\II \p \pb \phi$ on $X$ with tangent cone at infinity , which can be approximated by
				\begin{equation*}
					\pi^{*}\omega_{\CC}+\omega_{SRF},
				\end{equation*}
			outside a large compact set. Here $\omega_{SRF}$ denotes for a certain real $(1,1)$-form, which is Calabi-Yau when restricted on each fiber. 
		\end{theorem}
		\begin{remark}\label{S_1_Remark_1}
			This Calabi-Yau K\"ahler form is unique in the sense that, if two K\"ahler potentials are in a certain asymptotic class, and solve the same equation, then they at most differ by a pluriharmonic function, with slower-than-quadratic growth.
		\end{remark}
		\begin{remark}\label{S_1_Remark_Example_1}
			In this construction, the Laplacian admitting a bounded right inverse outside a large compact set is suffice for our purpose. As a consequence, we are able to produce warped QAC Calabi-Yau manifolds which are not covered in \cite{conlon2023warped}. For example
				\begin{equation*}
					X_{k,l}:z_{1}^{2}+z_{2}^{2}+z_{3}^{k}+z_{4}^{l}=1,2k>l,
				\end{equation*}
			where the normalized weighted action is
				\begin{equation*}
					a\cdot (z_{1},z_{2},z_{3})=(a^{k}z_{1},a^{k} z_{2},a^{2}z_{3}).
				\end{equation*}
			In this case $p=2k>l=q$, and therefore our theorem applies. On the other hand, this example corresponds to case $N=1$, $\beta=\max\{6,2k\}$, $m_{1}=3$ and $v=\frac{l}{2k}$ in \cite{conlon2023warped}. When $l$ is sufficiently closed to $2k$, for example $k=2024, l=4047$, $X_{k,l}$ violates the numerical constraint in Theorem 1.4 therein.
		\end{remark}
		\begin{remark}\label{S_1_Remark_Example_2}
			In a slightly different vein, let $z_{3}$, instead of $z_{4}$, be the holomorphic fibration to $\CC$ and the Calabi-Yau cone $V_{0}'$ be given by $z_{1}^{2}+z_{2}^{2}+z_{4}^{l}=0$. When $k\neq l$, if we require further that $2l>k$, then there are a different kind of warped QAC Calabi-Yau metrics on $X_{k,l}$. On the other hand, Theorem 8.1 in \cite{collins2019sasaki} shows that
				\begin{equation*}
					z_{1}^{2}+z_{2}^{2}+z_{3}^{k}+z_{4}^{l}=0, 2l>k, 2k>l
				\end{equation*}
			admits a Calabi-Yau cone structure. Therefore there is an AC (asymptotically conical) Calabi-Yau metric on its smoothing $X_{k,l}$. To conclude, there are three different kinds of Calabi-Yau metrics on $X_{k,l}$:
				\begin{enumerate}
					\item Warped QAC, with $\CC\times \CC^{2}/\Z_{k}$ as tangent cone at infinity;
					\item Warped QAC, with $\CC\times \CC^{2}/\Z_{l}$ as tangent cone at infinity;
					\item AC, with $\text{Cone}\big(\sharp^{\gcd(k,l)-1}S^{2}\times S^{3}\big)$ as tangent cone at infinity.
				\end{enumerate}
		\end{remark}
	
	A key point to our construction in Theorem \ref{Thm_1} is that we optimize the algorithm in \cite{szekelyhidi2019degenerations} to give an optimal control of the Ricci potential for the approximate metric $\pi^{*}\omega_{\CC}+\omega_{SRF}$, which is crucial in the gluing theorem.\\
		\begin{example}\label{S_1_Example_1}
			To see the comparison of collapsing and colliding explicitly, consider the affine quadrics in $\CC^{n+2}$. Given positive real numbers $t,x$, define $P,Q$ by
				\begin{equation*}
					\begin{split}
						P(\tilde{z})&=z_{1}^{2}+\cdots+ z_{n+1}^{2},\\
						Q(z_{n+2})&=z_{n+2}^{2}-t^{2x}, x,t\in \R^{+}.
					\end{split}
				\end{equation*}
			Thus $P+Q=0$ is an affine quadric in $\CC^{n+2}$. The relevant weighted action of $\CC^{*}$ on $\CC^{n+2}$ is taken to be
				\begin{equation*}
					a\cdot (\tilde{z},z_{n+2})=(a^{\frac{n}{n-1}}z_{1},\cdots ,a^{\frac{n}{n-1}}z_{n+1},az_{n+2}),
				\end{equation*}
			and the degree $p$ is $\frac{2n}{n-1}$. Let $\varphi$ be a K\"ahler potential for the Stenzel metric on the standard affine quadric (Page 21, Example 3, \cite{conlon2013asymptotically}). Then there is a unique Calabi-Yau K\"ahler form $\II \p \pb\phi_{x,y}$ approximated by $\pi^{*}\omega_{\CC}+\omega_{SRF}$
				\begin{equation*}
					\II \p \pb \big(|z_{n+2}|^{2}+t^{-2y}|P(\tilde{z})|^{2/p}\varphi(P(\tilde{z})^{-1/p}\cdot \tilde{z})\big), y\in \R^{+}.
				\end{equation*}
			Here, $x$ measures the rate of colliding and $y$ measures the rate of collapsing. After the coordinate change
				\begin{equation*}
					\tilde{v}=t^{-2x/p}\cdot \tilde{z}, v_{n+2}=t^{-x}z_{n+2},
				\end{equation*}
			the defining equation becomes
				\begin{equation*}
					v_{1}^{2}+\cdots+v_{n+1}^{2}+v_{n+2}^{2}=1,
				\end{equation*}
			and the approximate K\"ahler form is
				\begin{equation*}
					t^{-2x}\II \p \pb \big(|v_{n+2}|^{2}+t^{2(\frac{p+1}{p}x-y)}|P(\tilde{v})|^{2/p}\varphi(P(\tilde{v})^{-1/p}\cdot \tilde{v})\big).
				\end{equation*}
			The sign of $\frac{p+1}{p}x-y$ measures the ratio of the rates of collapsing and colliding. If it is greater than, equal to, or smaller than $0$, it corresponds to the situation where collapsing is slower than, comparable to, or faster than colliding, respectively.
		\end{example}
		\begin{remark}\label{S_1_Remark_NonCollapse}
			Throughout this paper, the term "collapsing" refers to scaling down the metric in the fiber direction. In contrast, the volume of unit balls always have a positive lower bound. In fact, this one-parameter family of Calabi-Yau manifolds is non-collapsing in the usual sense.
		\end{remark}
	
	We can put the above discussion into a more general setting. Let $t\to \infty$ be the collapsing parameter, and suppose the singular locus of the holomorphic fibration converges to $d$ distinct points $\{s_{1},\cdots ,s_{d}\}$ on the base space. Consider a family of hypersurfaces
		\begin{equation}\label{S_1_Eqn_Collapsing_1}
			X_{t}': P(\tilde{z})+\prod_{i=1}^{d}\prod_{j=1}^{q_{i}}(z_{n+2}-s_{i}-a_{t}^{-\theta_{i}}s_{i,j})=0.
		\end{equation}
	Here $\sum q_{i}=q$ and $a_{t}\in \CC, |a_{t}|=t$. The factors $\theta_{i}=\frac{p-q}{p-q_{i}}>0$ are chosen so that in small neighborhood of critical points $S_{i}=(0,s_{i})$, $X_{t}'$ is biholomorphic to:
		\begin{equation*}
			X_{i}: P(\tilde{z})+Q_{i}(z_{n+2})=0,
		\end{equation*}	
	where $Q_{i}=\prod_{l\neq i}(s_{i}-s_{l})^{q_{l}}\prod_{j=1}^{q_{i}}(z_{n+2}-s_{i,j})$. Assume for simplicity that:
		\begin{equation*}
			q_{1}\leq \cdots \leq q_{d}.
		\end{equation*}
	Our second main result is:
		\begin{theorem}\label{Thm_2}
			Suppose $p>2q_{i}$ and if $q_{1}<\frac{1}{4}p<q_{d}$ suppose further that $\frac{p+4q_{1}}{5p-4q_{1}}>\frac{2q_{d}}{3p-2q_{d}}$. Then, when the collapsing parameter $t$ is sufficiently large, the unique Calabi-Yau K\"ahler form $\II \p \pb \phi_{t}$ on $X_{t}'$, which is approximated by
				\begin{equation*}
					\pi^{*}\omega_{\CC}+t^{2(q-p)/p}\omega_{SRF},
				\end{equation*}
			in the sense of Theorem \ref{Thm_1}, can be obtained by a gluing construction. In particular, there is a family of small neighborhoods $U_{i,t}$ around $S_{i}\in \CC^{n+2}$ together with a family of biholomorphism 
				\begin{equation*}
					F_{i,t}: U_{i,t}\cap X_{t}' \to X_{i},
				\end{equation*}
			whose image exhausts $X_{i}$ and such that scaled Calabi-Yau K\"ahler forms $(F_{i,t}^{-1})^{*}(t^{2\theta_{i}}\II \p \pb \phi_{t})$ converge to the warped QAC Calabi-Yau K\"ahler forms on $X_{i}$ constructed in Theorem \ref{Thm_1}.
		\end{theorem}
		\begin{remark}\label{S_1_Remark}
			When there is only one singular fiber, collapsing typically amounts to rescaling. To see this, let $X_{t}'\simeq \CC^{n+1}$ be given by
				\begin{equation*}
					P(\tilde{z})+z_{n+2}=0,
				\end{equation*}
			and consider two coordinates related by the scaling
				\begin{equation*}
					(\tilde{z},z_{n+2})=(t^{1/p}\cdot \tilde{v},tv_{n+2}).
				\end{equation*}
				The approximate solution $\pi^{*}\omega_{\CC}+t^{2(1-p)/p}\omega_{SRF}$ in the coordinate $(\tilde{z},z_{n+2})$, can be represented as $t^{2}(\pi^{*}\omega_{\CC}+\omega_{SRF})$ in the coordinate $(\tilde{v},v_{n+2})$.
		\end{remark}
		\begin{example}\label{S_1_Example_2}
			This is a continuation for the Example \ref{S_1_Example_1}. Let $x=0, y=\frac{1}{n}$, and in this case, collapsing is faster than colliding. Then when $t$ is sufficiently large, in small neighborhoods of $S_{1}=(0,1), S_{2}=(0,-1)$, the metrics are modeled on $(\CC^{n+1},t^{-\frac{2}{n+1}}\omega_{\CC^{n+1}})$, where $\omega_{\CC^{n+1}}$ is the unique warped QAC Calabi-Yau K\"ahler form having tangent cone $\CC\times A_{1}$ at infinity.
		\end{example}
		\begin{remark}\label{S_1_Remark_pGH}
			Using this gluing construction, we are able to produce a sequence of non-collapse (in the usual sense) warped QAC Calabi-Yau manifolds $(M_{i},\omega_{i})$. With an appropriate choice of marked points $p_{i}$ and rescaling factors $\lambda_{i}$, the pointed sequence
				\begin{equation*}
					(M_{i},p_{i},\lambda_{i} \omega_{i}),
				\end{equation*}
				converges in pointed-Gromov-Hausdorff sense to a smooth pointed warped QAC Calabi-Yau manifold
					\begin{equation*}
						(M_{\infty}, p_{\infty},\omega_{\infty}),
					\end{equation*}
				whose tangent cone at infinity is $\CC\times V_{0}$. The limit $(M_{\infty},p_{\infty},\omega_{\infty})$ is referred to as a minimal bubble in \cite{sun2023bubbling}, and its geometry depends sensitively on the choice of $p_{i},\lambda_{i}$ (See Proposition \ref{S_5_Prop_Bubble_Convergence} for more detailed discussion on bubble tree convergence).
				\end{remark}
		
	Our ansatz for the collapsing metrics in many respect, is simpler than the one described by Yang Li in \cite{li2019gluing}, since, outside a compact set, the metric on each fiber is already Calabi-Yau, and in a certain fixed compact region, the behavior of the metric is very similar. In this compact region we have a semi-Ricci-flat description away from the singular fiber. On the other hand near the critical points, we will need to use a rather coarse function spaces to handle the error introduced by the biholomorphisms.\\
	
	However, the difference in the case considered here with the case in a compact manifold is that we will need to understand the geometry of the collapsing at infinity, which will not happen in the compact case. As a toy model, consider a sufficiently large sphere in $(\CC\times V_{1}, g_{\CC}\oplus t^{-2}g_{V_{1}})$, where $(V_{1},g_{V_{1}})$ is asymptotic to a cone $(V_{0},g_{V_{0}})$. As $t$ goes to infinity, the sphere will develop a circle family of singularities modeled on $V_{0}$. The key feature of our construction is that the weighted H\"older spaces are chosen to provide good control of Green's operator of various model geometries.
	
	\subsection*{Guide to the proof}
		\begin{enumerate}
			\item In Section \ref{S_2} we modify Sz\'ekelyhidi's approach in \cite{szekelyhidi2019degenerations} and give a simple proof of the existence of the Calabi-Yau metrics on the affine hypersurface $P(\tilde{z})+Q(z_{n+2})=0$. In Proposition \ref{S_2_Prop_Ricci_1}, we use a different approach to compute the Ricci potential of the approximate solution, obtaining an optimal estimate for its decay rate. Proposition \ref{S_2_Prop_Map_1}, \ref{S_2_Prop_Map_2}, \ref{S_2_Prop_Map_3} and \ref{S_2_Prop_Map_4} establish the elliptic theory for the Laplacian for each model geometry.
			\item Section \ref{S_3} is devoted to analyzing the geometry of our approximate model for the collapsing problem. We define the $t$-dependent weighted H\"older spaces $C^{k,\alpha}_{\bs{\delta},\tau,t}$, in terms of the weighted functions adapted to the collapsing problem. Section \ref{S_3_2} estimates the Ricci potential in terms of weighted H\"older norm, and compare the approximate metrics with every relevant model spaces at different scale (see Lemma \ref{S_3_Lemma_Met_1}, \ref{S_3_Lemma_Met_2}, \ref{S_3_Lemma_Met_3} and \ref{S_3_Lemma_RicciMet_4}). The technical part of this section is the proof of Proposition \ref{S_3_Prop_Ricci}, which compares $C^{0,\alpha}_{\bs{\delta-2},\tau-2,t}$ with $C^{0,\alpha}_{\bs{0},0,t}$. Assuming $p,q_{i}$ are constraint as in Theorem \ref{Thm_2}, we are able to show that the error between the approximate solution and the Calabi-Yau metric obtained by gluing construction is small in a usual sense.
			\item In Section \ref{S_4}, we construct a "paramatrix" $P$ with domain $C^{0,\alpha}_{\bs{\delta-2},\tau-2,t}$ to a direct sum of H\"older spaces $C^{2,\alpha}_{\text{model, }\bs{\delta},\tau,t}$, by gluing the Green's operators of each scaled model spaces. The difference between $\Delta P$ and $Id$ is small as shown in Proposition \ref{S_4_Prop_RI}. In the end, we prove the gluing theorem in \ref{S_4_Thm_Gluing} using Banach fixed point theorem.
			\item In Section \ref{S_5} we discuss a bubble tree structure from this gluing construction.
		\end{enumerate}
	\subsection*{Notation}
		Throughout this paper the notation $a\lesssim b$ is used to indicate $a<C b$ for some uniform constant $C$.

\section{Construction and properties of Warped QAC-CY metric}\label{S_2}
	
	Recall the hypersurface $X$ is given by
		\begin{equation*}
			P(\tilde{z})+Q(z_{n+2})=0.
		\end{equation*}
	The key to our construction is that, every smooth fiber of the holomorphic fibration is biholomorphism to $V_{1}\subset \CC^{n+1}$, where
		\begin{equation*}
			V_{1}: P(\tilde{z})+1=0.
		\end{equation*}
	The $(1,1)$-form $\omega_{SRF}$ is obtained by pulling back a certain Calabi-Yau K\"ahler form on $V_{1}$. In \cite{conlon2013asymptotically} Conlon and Hein proved the existence of Calabi-Yau K\"ahler $\II \p \pb \varphi$, such that $\varphi$ is asymptotic to $r^{2}$ on $V_{0}$ at large distance. There are two ways to justify this comparison. The first one is to extend $r^{2}$ to a neighborhood of $V_{0}$ by the homogeneity of $r^{2}$ under the weighted action; the second one is to construct a map from $V_{1}$ to $V_{0}$: Let $G_{0}$ be the geodesy projection from $\CC^{n+1}$ to $V_{0}$ under a singular metric $\II \p \pb R^{2}$. Here $R$ equals to $1$ on the unit sphere in $\CC^{n+1}$, and is homogenous under weighted action. Outside a large compact set in $V_{1}$, $G_{0}$ is a diffeomorphism into an open set in $V_{0}$. There is a positive constant $c$ such that
		\begin{equation*}
			|\nabla^{k}(\varphi-G_{0}^{*}(r^{2}))|_{g_{\varphi}}<C_{k}r^{-c-k}.
		\end{equation*}
	When $P(\tilde{z})\neq 0$, choose a branch of $P(\tilde{z})^{-1/p}$ and define
		\begin{equation*}
			\omega_{SRF}=\II \p \pb |P(\tilde{z})|^{2/p}\varphi(P(\tilde{z})^{-1/p}\cdot \tilde{z})
		\end{equation*}
	This definition is independent on the choice of the branch. Although 
		\begin{equation*}
			|P(\tilde{z})|^{2/p}\varphi(P(\tilde{z})^{-1/p}\cdot \tilde{z}),
		\end{equation*}
	is not defined when $P(\tilde{z})=0$, we can extend it smoothly to $r^{2}$ outside a large compact set. We will use $|P(\tilde{z})|^{\frac{2}{p}}\varphi(P(\tilde{z})^{\frac{-1}{p}}\cdot \tilde{z})$ to denote this extension unambiguously.\\

	Let $\phi'$ be the K\"ahler potential defined as follow
		\begin{equation}\label{S_2_Defn_AppSol}
			\phi'=\widetilde{max}(C(1+|z|^{2})^{a},|z_{n+2}|^{2}+|P(\tilde{z})|^{2/p}\varphi(P(\tilde{z})^{-1/p}\cdot \tilde{z})).
		\end{equation}
	Here, $\widetilde{\max}$ is the Demailly regularized maximum in \cite{demailly1997complex}. We choose $C$ large enough and $a<w_{\min}$. One can check that $\phi'$ equals to $C(1+|z|^{2})^{a}$ in a compact set containing the critical points and to $|z_{n+2}|^{2}+|P(\tilde{z})|^{2/p}\varphi(P(\tilde{z})^{-1/p}\cdot \tilde{z}))$ outside a large compact set.\\
	
	 To state the result of the existence and the asymptotic behavior of the warped QAC Calabi-Yau metric in Theorem \ref{Thm_1}, we will need to define the weighted H\"older space as in \cite{szekelyhidi2019degenerations}. We first define the weight functions, which depend only on the distance functions in base space and fiber direction. We consider $R$ as the fiberwise distance, since $R$ and $r$ are uniformly equivalent. Let $\rho^{2}=|z_{n+2}|^{2}+R^{2}$ be an approximation for the distance function. Define $w$ to be a smooth function satisfying the following properties
			\begin{equation}\label{S_2_Defn_W}
				w=
					\begin{cases}
						1, & R>2\kappa\rho\\
						\frac{R}{\kappa\rho}, & 2\kappa^{-1}\rho^{q/p}<R<\kappa \rho\\
						\kappa^{-2}\rho^{q/p-1}, & R<\kappa^{-1}\rho^{q/p}
					\end{cases}.
			\end{equation}
		Let $A'$ be a sufficiently large constant, and define
			\begin{equation}\label{S_2_Defn_Holder_1}
				\|T\|_{C^{k,\alpha}_{\delta,\tau}(X)}=\|T\|_{C^{k,\alpha}(\{\rho<2A'\})}+\|T\|_{C^{k,\alpha}_{\delta,\tau}(\{\rho>A'\})},
			\end{equation}
		where
			\begin{equation}\label{S_2_Defn_Holder_2}
				\begin{split}
					&\|T\|_{C^{k,\alpha}_{\delta,\tau}}=\sum_{j=1}^{k}\sup_{\rho>A'}\rho^{-\delta+k}w^{-\tau+k}|\nabla^{k}T|+[\rho^{-\delta+k}w^{-\tau+k}\nabla^{k}T]_{0,\alpha},\\
					&[T]_{0,\alpha}=\sup_{\rho(z)>A'}\rho(z)^{\alpha}w(z)^{\alpha}\sup_{z'\in B_{c}(z), z'\neq z}\frac{|T(z')-T(z)|}{d(z',z)^{\alpha}}.
				\end{split}
			\end{equation}
			\begin{remark}
				These weight functions are defined based on the geometry of the space. In the region $R>\Lambda^{3/4}\rho^{q/p}$ the metric is approximated by $\CC\times V_{0}$, and is approximated by $|z_{0}^{i}|^{q/p}g_{\CC\times V_{1}}$ in the region $\{|z_{n+2}-z_{0}^{i}|<3B|z_{0}^{i}|^{q/p}, R<2\Lambda^{2}\rho^{q/p},\rho>A\}$, (See Proposition \ref{S_2_Prop_Met_1}). $\CC\times V_{0}$ is a conical space admitting a singular link $L$. Roughly speaking we can use polar coordinate to write down the Laplacian
				\begin{equation*}
					\Delta_{\CC\times V_{0}}=\p_{\rho}^{2}+\frac{2n+1}{\rho}\p_{\rho}+\frac{1}{\rho^{2}}\Delta_{L}.
				\end{equation*}
			Moreover $L$ has a circle $S^{1}$ of singularities modeled on cone $V_{0}$. The weight function $w$ is equivalent to the distance function to the $S^{1}$, and we can apply \cite{rafe1991elliptic} to obtain a weighted Fredholm theory for the Laplacian operator $\Delta_{L}$. Moreover we can define the double weighted H\"older space $C^{k,\alpha}_{\delta,\tau}(\CC\times V_{0})$ similar to \ref{S_2_Defn_Holder_3}. The only difference is that we now replace $w$ by $w'$
				\begin{equation*}
					w'=
						\begin{cases}
							1, & R>2\kappa\rho\\
							\frac{R}{\kappa\rho}, & R<\kappa\rho
						\end{cases}.
				\end{equation*}
			\end{remark}

		Meanwhile the product $\rho w$ is a scaled distance function on the fibers, and it fits into the elliptic theory for $\CC\times V_{1}$. To be exact, $\rho w=\kappa^{-1}\xi$, where $\xi$ is a smoothing of $\max(R_{\tilde{z}},\kappa^{-1})$ on $V_{1}\subset \CC^{n+1}$. Define the weighted H\"older space $C^{k,\alpha}_{\tau}$ on $\CC\times V_{1}$ accordingly by
				\begin{equation}\label{S_2_Defn_Holder_3}
					\|T\|_{C^{k,\alpha}_{\tau}(\CC\times V_{1})}=\|\xi^{-\tau}T\|_{C^{k,\alpha}(\xi^{-2}g_{\CC\times V_{1}})}.
				\end{equation}  
		In the region $\{|z_{n+2}-z_{0}^{i}|<3B|z_{0}^{i}|^{q/p}, R<2\Lambda^{2}\rho^{q/p},\rho>A\}$ the weighted H\"older space coincides with the H\"older space $C^{k,\alpha}_{\delta,\tau}(X)$ up to a scaling (See Proposition \ref{S_2_Prop_Map_2}). Moreover we can define $C^{k,\alpha}_{\tau}(\CC\times V_{0})$ similarly by replacing $\xi$ by $R_{\tilde{z}}$.\\
			
		The precise statement for Theorem \ref{Thm_1} is:
			\begin{theorem}\label{Thm_1_Detail}
				Suppose $\delta\in (\frac{2(2q-p)}{p}, \frac{2q}{p})$ and $\tau$ is sufficiently closed to $0$, then $X$ admits a Calabi-Yau K\"ahler form $\II \p \pb \phi$, such that the difference between the corresponding K\"ahler potential $\phi$ and the approximate solution $\phi'$ in \ref{S_2_Defn_AppSol} satisfies
					\begin{equation}\label{S_2_Property_Diff}
						\varphi-\varphi' \in C^{2,\alpha}_{\delta,\tau}(X).
					\end{equation}
				Moreover $\phi$ is unique in the sense that if
					\begin{equation*}
						(\II \p \pb\phi)^{n+1}=(\II\p \pb (\phi+u))^{n+1},
					\end{equation*} 
				and $u\in C^{2,\alpha}_{\delta',\tau}(X)$ for some $\delta'<\frac{2q}{p}$ , then $u$ must be pluriharmonic.
			\end{theorem}
		The estimate in \ref{S_2_Property_Diff} is very close to be optimal. In the case for $(\CC^{3},\omega_{\CC^{3}})$, we recover the decay rate of Ricci potential in \cite{li2019new}. The main difference is, in \cite{li2019new}, inverting the Laplacian is reduced to approximately solving ODEs in the fiber direction. In our case, however, indicial roots must be avoided to obtain a Fredholm theory on the model space $\CC\times V_{0}$, but this difference is negligible.\\
		
		In Section \ref{S_2_1} we estimate the decay rate of the Ricci potential, and the difference between $\II \p \pb \phi'$ and the metric in the model spaces. Combine with the properties of the Laplacian on each model spaces developed in \cite{szekelyhidi2019degenerations}, in Section \ref{S_2_2}, we are able to invert the Monege Amp\`ere operator outside a large compact set. Therefore, the approximate solution $\phi'$ can be perturbed into $\phi''$ which is Calabi-Yau outside a large compact set. Then we use Hein's approach to construct a solution to the Monge Amp\`ere equation, and improve its decay rate.
	
	\subsection{Geometry of the approximate solution}\label{S_2_1}
		Our approach to compute the Ricci potential is different than that of \cite{szekelyhidi2019degenerations}, and philosophically similar to the one in \cite{li2019new}. Use binomial formula to expand the volume form
			\begin{equation*}
				(\II \p \pb \phi')^{n+1}=(n+1)\omega_{\CC}\wedge \omega_{SRF}^{n}+\omega_{SRF}^{n+1}.
			\end{equation*}
		Apply adjunction formula to transfer the computation on $\Omega^{n+1,n+1}_{X}$ to the pull back bundle $\Omega^{n+2,n+2}_{\CC^{n+2}}\big|_{X}$
			\begin{equation*}
				\Omega_{X}\wedge \ud (P+Q)=c_{n}\ud z_{1}\wedge \cdots \wedge \ud z_{n+2}.
			\end{equation*}
		Using the fact that $\phi'$ is Calabi-Yau on each fibers, the term
			\begin{equation*}
				(n+1)\omega_{\CC}\wedge \omega_{SRF}^{n}\wedge  \II \ud P\wedge \ud \overline{P}=c_{n}'\II \ud z_{1}\wedge \cdots \wedge \ud \overline{z_{n+2}}
			\end{equation*}
		is a constant, and this enable us to improve decay rate of the Ricci potential $h_{Ric}$. The only non-constant term for $h_{Ric}$ is contributed by the second term
			\begin{equation*}
				 \big(\II \p \pb|P(\tilde{z})|^{2/p}\varphi(P(\tilde{z})^{-1/p}\cdot \tilde{z})\big)^{n+1}\wedge \II \ud Q\wedge \ud \overline{Q}.
			\end{equation*}
		Normalize $\Omega_{X}$ so that the Ricci potential $h_{Ric}=\log(1+c_{n}''h)$ where
			\begin{equation*}
				h=\frac{\big(\II \p \pb|P(\tilde{z})|^{2/d}\varphi(P(\tilde{z})^{-1/p}\cdot \tilde{z})\big)^{n+1}}{\II \ud z_{1}\wedge \cdots \wedge\ud \overline{z_{n+2}}}
			\end{equation*}
 
			\begin{prop}\label{S_2_Prop_Ricci_1}
				The Ricci potential for the approximate solution $h\in C^{k,\alpha}_{\delta-2,\tau-2}(X)$ for some $\delta\in (\frac{2(2q-p)}{p},\frac{2q}{p})$. 
			\end{prop}
			
			\begin{proof}
				The idea is similar to \cite{szekelyhidi2019degenerations}, but the argument is simpler. The strategy is to use scaled coordinate to compare $X$ with the model geometries. And since we have to deal with difference variables, a subscript will be added to indicate the dependency on the variables. Denote $Err:=|P(\tilde{z})|^{2/p}\varphi(P(\tilde{z})^{-1/p}\cdot \tilde{z})-r^{2}_{\tilde{z}}$. We can rewrite $h$ as follow:
					\begin{equation*}
						h=\frac{(\II \p \pb(r_{\tilde{z}}^{2}+Err))^{n+1}\wedge \II \ud Q\wedge \ud \overline{Q}}{\II \ud z_{1}\wedge \cdots \wedge \ud \overline{z_{n+2}}}.
					\end{equation*}
					\subsubsection*{$R_{\tilde{z}}>\kappa \rho$ :}
						Consider the coordinate change
							\begin{equation*}
								\tilde{z}=D\cdot \tilde{v}, z_{n+2}= D v_{n+2}, D<\rho<2D.
							\end{equation*}
						Now in this coordinate
							\begin{equation*}
								|\nabla^{k}_{D^{-2}g_{\phi'}}D^{-2}Err|<C_{k}D^{(2+c)(q-p)/p}.
							\end{equation*}
						Write $Q_{D}(v_{n+2})=D^{-p}Q$ for short. The Ricci potential can be estimated
							\begin{equation*}
								|\nabla_{D^{-2}g_{\phi'}}h|=D^{2(q-p)}\bigg|\nabla^{k}\frac{(\II\p \pb(r_{\tilde{v}}^{2}+D^{-2}Err))^{n+1}\wedge \II \ud Q_{D}\wedge \overline{Q_{D}}}{\II \ud v_{1}\wedge \cdots \ud \overline{v_{n+2}}}\bigg|<C_{k}D^{2(q-p)}.
							\end{equation*}
					
					\subsubsection*{$\kappa^{-1}\rho^{q/p}<R_{\tilde{z}}<\kappa\rho$ :}
						Consider the coordinate change
							\begin{equation*}
								\tilde{z}=K\cdot \tilde{v}, z_{n+2}=z_{n+2}^{0}+Kv_{n+2}, |v_{n+2}|<1,
							\end{equation*}
						with $D<\rho<2D, K<R_{\tilde{z}}<2K$ and $\kappa^{-1}\rho^{q/p}<K<\frac{1}{2}\kappa\rho$. Similar argument implies
							\begin{equation*}
								\begin{split}
									&|\nabla_{K^{-2}g_{\phi'}}^{k}K^{-2}Err|<C_{k}(K^{-1}D^{q/p})^{2+c},\\
									&|\nabla_{K^{-2}g_{\phi'}}h|<C_{k}D^{2(q-1)}K^{-2(p-1)}.
								\end{split}
							\end{equation*}
					\subsubsection*{$R_{\tilde{z}}<\kappa^{-1}\rho^{q/p}$ :}
						Consider coordinate change
							\begin{equation*}
								Q(z_{n+2})=v_{n+2}^{q}, \tilde{z}=v_{n+2}^{q/p}\cdot \tilde{v}.
							\end{equation*}
						The defining equation is
							\begin{equation*}
								P(\tilde{v})+1=0, R_{\tilde{v}}\lesssim \kappa^{-1},
							\end{equation*}
						with potential
							\begin{equation*}
								|Q_{t}^{-1}(v_{n+2}^{1/p})|^{2}+|v_{n+2}|^{2q/p}\varphi(\tilde{v}).
							\end{equation*}
						In the region $\{R_{\tilde{v}}\lesssim \kappa^{-1}, |v_{n+2}-v_{0}|<B|v_{0}|^{q/p}\}$, the metric is approximated by $|v_{0}|^{2q/p}g_{\CC\times V_{1}}$. It follows immediately that
							\begin{equation*}
								|\nabla^{k}_{g_{\phi'}} h|=C_{k}(\kappa) D^{2(q-p)/p-kq/p}
							\end{equation*}
						We conclude that
							\begin{equation*}
								|\nabla_{g_{\phi'}}^{k}h|<C_{k}\rho^{\delta-k}w^{\tau-k},
							\end{equation*}
						for some $\delta>\frac{2(2q-p)}{p}$, $\tau<0$ and sufficiently close to $0$.
			\end{proof}
			\begin{prop}\label{S_2_Prop_Met_1}
				For any $\epsilon>0$, there exists $A(\epsilon),\Lambda(\epsilon)$ such that $\forall A>A(\epsilon), \Lambda>\Lambda(\epsilon)$, then on the region $\{\rho>A, R_{\tilde{z}}>\Lambda^{3/4} \rho^{q/p}\}$
					\begin{equation*}
						\|g_{\CC\times V_{0}}-g_{\phi'}\|_{C^{k,\alpha}_{0,0}}<C_{k}\epsilon.
					\end{equation*}
				Moreover, on the region $U_{i}:=\{|z_{n+2}-z_{0}^{i}|<3B|z_{0}^{i}|^{q/p}, R_{\tilde{z}}<2\Lambda^{2}\rho^{q/p},\rho>A\}$, we can first choose $\Lambda$ large enough and then choose $B$ large enough and at last $A$ large enough, then we have
					\begin{equation*}
						\||z_{0}^{i}|^{2q/p}g_{\CC\times V_{1}}-g_{\phi'}\|_{C^{k,\alpha}_{0}(U_{i})}<C_{k}\epsilon.
					\end{equation*}
			\end{prop}
			\begin{proof}
				The proof is the same as in Lemma \ref{S_3_Lemma_Met_1} by setting $t=1$.
			\end{proof}
		
	\subsection{Mapping properties of Laplacian}	\label{S_2_2}
		We mainly utilize the following propositions to define the paramatrixes both in constructing warped QAC metric and in the gluing theorem. The proofs are drawn down from \cite{szekelyhidi2019degenerations}. It should be remarked that in our case the dimension of the manifold is $2n+2$ while in \cite{szekelyhidi2019degenerations} the dimension is $2n$.
			\begin{prop}\label{S_2_Prop_Map_1}
				Let $\tau\in (2-2n,0)$, the Laplacian
				\begin{equation*}
					\Delta_{\CC\times V_{0}}: C^{2,\alpha}_{\tau}(\CC\times V_{0})\to C^{0,\alpha}_{\tau-2}(\CC\times V_{0})
				\end{equation*}
				is invertible. Denotes its inverse as $P_{3}$.
			\end{prop}
			\begin{proof}
				See Corollary 12 and the proof is identical to Proposition 19 in \cite{szekelyhidi2019degenerations}.
			\end{proof}
			\begin{prop}\label{S_2_Prop_Map_2}
				Let $\tau\in (2-2n,0)$ and $\delta_{0}$ avoid a discrete set of indicial roots, the Laplacian
				\begin{equation*}
					\Delta_{\CC\times V_{0}}: C^{2,\alpha}_{\delta_{0},\tau}\to C^{0,\alpha}_{\delta_{0}-2,\tau-2}(\CC\times V_{0})
				\end{equation*}
				is invertible. Denote its inverse as $P_{0}$.
			\end{prop}
			\begin{proof}
				See Proposition 13, 23 in \cite{szekelyhidi2019degenerations}.
			\end{proof}
		
			\begin{prop}\label{S_2_Prop_Map_3}
				Let $\tau$ as above, the Laplacian
				\begin{equation*}
					\Delta_{\CC\times V_{1}}: C^{2,\alpha}_{\tau}(\CC\times V_{1})\to C^{0,\alpha}_{\tau-2}(\CC\times V_{1})
				\end{equation*}
				is invertible. Denote its inverse as $P_{1}$.
			\end{prop}
			\begin{proof}
				See Proposition 21,  24 in \cite{szekelyhidi2019degenerations}.
			\end{proof}		
			\begin{prop}\label{S_2_Prop_Map_4}[c.f. \cite{szekelyhidi2019degenerations}]
				There is a bounded right inverse for the Laplacian $\Delta_{g_{\phi'}}$ on $\{\rho>A\}$ for some sufficiently large $A$.
			\end{prop}
			\begin{proof}
				The proof is very similar to that of lemma \ref{S_4_Lemma_Laplacian_1}, we sketch the proof here. Let $\gamma_{i},i=1,2$ be smooth function on $\R$ such that $\gamma_{1}+\gamma_{2}=1$ and
					\begin{equation*}
						\gamma_{1}(s)=
							\begin{cases}
								1 & s>2\\
								0 & s<1
							\end{cases}.
					\end{equation*}
				Define the cutoff functions
					\begin{equation*}
						\begin{split}
							&\beta_{i}=\gamma_{i}(R_{\tilde{z}}\Lambda^{-1}\rho^{-q/p}), i=1,2;\\
							&\widetilde{\beta_{1}}=\gamma_{1}\bigg(\frac{\ln(R_{\tilde{z}}\Lambda^{-1/2}\rho^{-q/p})}{\ln(\Lambda^{1/4})}\bigg);\\
							&\widetilde{\beta_{2}}=\gamma_{2}\bigg(\frac{\ln(R_{\tilde{z}}2^{-1}\rho^{-q/p})}{\ln(\Lambda)}\bigg).
						\end{split}
					\end{equation*}
				And let $\chi_{i},\widetilde{\chi_{i}}$ be cutoff functions such that
					\begin{equation*}
						\begin{split}
							&\chi_{i}\equiv 1 \text{ on } |z_{n+2}-z_{0}^{i}|<B|z_{0}^{i}|^{q/p};\\
							&\chi_{i}\equiv 0 \text{ on } |z_{n+2}-z_{0}^{i}|>2B|z_{0}^{i}|^{q/p};\\
							&\widetilde{\chi_{i}}\equiv 1 \text{ on } |z_{n+2}-z_{0}^{i}|<2B|z_{0}^{i}|^{q/p};\\
							&\widetilde{\chi_{i}}\equiv 0 \text{ on } |z_{n+2}-z_{0}^{i}|>3B|z_{0}^{i}|^{q/p}.
						\end{split}
					\end{equation*}
				Define the approximate right inverse as in Proposition 23 in \cite{szekelyhidi2019degenerations}
					\begin{equation*}
						P'(f)=\widetilde{\beta_{1}}P_{0}(\beta_{1}f)+\sum_{i}\widetilde{\chi_{i}}|z_{0}^{i}|^{2q/p}P_{1}(\chi_{i}\beta_{2}f).
					\end{equation*}
				And then use Proposition \ref{S_2_Prop_Met_1} to show that $\Delta P'$ is close to identity as operator, we can then define the right inverse $P=P'(\Delta P')^{-1}$ on the region $\{\rho>A\}$.
			\end{proof}

	\subsection{Perturb into Calabi-Yau metric}\label{S_2_3}
		On the region $\{\rho>A\}$ we can linearize the Monge-Amp\`ere operator
			\begin{equation*}
				\mathscr{F}: C^{0,\alpha}_{\delta-2,\tau-2}(\{\rho>A\})\to C^{0,\alpha}_{\delta-2,\tau-2}(\{\rho>A\}), f\mapsto\frac{(\omega_{\phi'}+\II \p \pb Pf)}{\II^{(n+1)^{2}}\Omega_{X}\wedge \overline{\Omega_{X}}}-1.
			\end{equation*}
		Once $A$ is sufficiently large, there exist an unique solution $f$ such that $\mathscr{F}(f)=0$ and $\|f\|_{C^{0,\alpha}_{\delta-2,\tau-2}}<\epsilon_{0}$, for some sufficiently small $\epsilon_{0}$. Now we define another K\"ahler potential $\phi''$ such that it is exactly Calabi-Yau on the region $\{\rho>2A\}$
			\begin{equation*}
				\phi''=\phi'+\gamma_{1}(\rho/(2A))Pf
			\end{equation*}
			\begin{lemma}\label{S_2_Lemma_SOB}
				$(X,g_{\phi''})$ is of $SOB(2n+2)$, and admits a $C^{3,\alpha}$ quasi-altlas.
			\end{lemma}
			\begin{proof}
				The proof is identical to Page 2691 in \cite{szekelyhidi2019degenerations}. The idea for $SOB(2n+2)$ is that the model space $\CC\times V_{0}$ takes up most of the volume in $X$, and $X$ admitting a $C^{3,\alpha}$ charts follows naturally by estimating the biholomorphisim and the metric on each model geometries.
			\end{proof}
		
		The K\"ahler form $\omega_{\phi''}$ is Calabi-Yau on $\{\rho>2A\}$, and therefore the Ricci potential has compact support. Use the technique in Chapter 4 \cite{hein2010gravitational}, we are able to produce a solution to the Monge-Amp\`ere equation
			\begin{equation}\label{S_2_Eqn_MA}
				(\omega_{\phi''}+\II \p \pb u)^{n+1}=e^{h_{Ric}}\omega_{\phi''}^{n+1}, u \in C^{4,\alpha}(X),
			\end{equation}
		by taking a convergent subsequence of $u_{\epsilon}$, which solve $\epsilon$-perturbed equation
			\begin{equation}\label{S_2_Eqn_PMA}
				(\omega_{\phi''}+\II \p \pb u_{\epsilon})^{n+1}=e^{h_{Ric}+\epsilon u_{\epsilon}}\omega_{\phi''}^{n+1}.
			\end{equation}
		The key to the construction is to derive a uniform $C^{0}(X)$ bound for $|u_{\epsilon}|$, then uniform $C^{4,\alpha}(X)$ follows from the local estimate. However, to prove Theorem \ref{Thm_1_Detail}, a better estimate for $|u|$ is needed.
			\begin{lemma}\label{S_2_Lemma_MADecay}
				Suppose the manifold $X$ is of $SOB(\beta),\beta>2$. If $u$ is a solution to the Monge-Amp\`ere equation \ref{S_2_Eqn_MA} constructed in \cite{hein2010gravitational}, and $h$ has a compact support with $\sup \rho^{\mu}|h|<C_{0}$ for some $2<\mu<\beta$, then for every small positive constant $\epsilon'$, there exist a constant $C(\epsilon')$ such that
					\begin{equation*}
						|u|<C(\epsilon', C_{0})\rho^{2-\mu+\epsilon'}.
					\end{equation*}
			\end{lemma}
			\begin{proof}
				Write $\omega_{0}=\omega_{\phi''}, \omega_{f}=\omega_{\phi''+f}$ for simplicity, and $g=g_{\omega_{0}}$ be the background metric. Modify the distance function $\rho=\sqrt{|z_{n+2}|^{2}+R_{\tilde{z}}^{2}}$ in a compact set, so that the modified function is greater than $1$. Also by definition  $|\nabla \rho|+\rho|\nabla^{2}\rho|<C$.  Theorem 1.2 in \cite{hein2011weighted} implies: there exists a constant $C$ and $s\in [1,\frac{2n+2}{2n}]$
					\begin{equation}\label{S_2_Eqn_WeightedPoincare}
						\bigg(\int |f|^{2s}\rho^{s(\beta-2)-\beta}\bigg)^{\frac{1}{s}}<C\int |\nabla f|^{2}.
					\end{equation}
				The difference between the volume forms $\omega_{u_{\epsilon}}^{n+1}$ and $\omega_{0}^{n+1}$ is
					\begin{equation}\label{S_2_Eqn_VolumeDifference}
						(e^{h+\epsilon u_{\epsilon}}-1)\omega_{0}^{n+1}	=\II \p \pb u_{\epsilon}\wedge T=\frac{1}{2}\ud \ud ^{c}u_{\epsilon}\wedge T.
					\end{equation}
				where $T=\omega_{0}^{n+1}+\omega_{0}^{n}\wedge \omega_{u_{\epsilon}}+\cdots \omega_{u_{\epsilon}}^{n+1}$, and $\ud ^{c}=J\ud $. Let $\xi\in C^{\infty}_{0}(X)$ and $\zeta=\rho^{m}$ for some $m \in \R$.  Multiply \ref{S_2_Eqn_VolumeDifference} by $\xi(\zeta^{2} u_{\epsilon}|\zeta u_{\epsilon}|^{d-2}), d>1$ and integrate by part
					\begin{equation}\label{S_2_Ineqn_Int_1}
						\begin{split}
							&\int \xi|\nabla|\zeta u_{\epsilon}|^{\frac{d}{2}}|^{2}+\frac{(n+1)d^{2}}{2(d-1)}\int \xi\big(\zeta(e^{\epsilon u_{\epsilon}}-1)\zeta u_{\epsilon}e^{h}\big)\big|\zeta u_{\epsilon}\big|^{d-2}\\
							&< \frac{(n+1)d^{2}}{4(d-1)}\bigg(\int \xi\big|\zeta(e^{h}-1)\big|\big|\zeta u_{\epsilon}\big|^{d-1}+\frac{C|m|(|m|+1)d}{(d-1)}\int \xi\rho^{-2}|\zeta u_{\epsilon}|^{d}\\
							&+\frac{C(d-1)}{8d^{2}}\bigg|\int \zeta u_{\epsilon} |\zeta u_{\epsilon}|^{d-2}\ud ^{c}\zeta u_{\epsilon}\wedge \ud \xi\wedge T\bigg|\bigg).
						\end{split}
					\end{equation}
				Fix $\epsilon$ and $k\in \R$ sufficiently large and apply $\xi=\rho^{-k}\gamma_{2}(\rho/M), m=0$ to \ref{S_2_Ineqn_Int_1} and then let $M\to \infty$. Use the fact that $\|u_{\epsilon}\|_{C^{2}}$ is uniformly bounded, we conclude
					\begin{equation*}
						\int \rho^{-k}|u_{\epsilon}|^{d-2}|\nabla u_{\epsilon}|^{2}+\epsilon\rho^{-k}|u_{\epsilon}|^{d}< C_{k}(d)\bigg(\int \rho^{-k-2}|u_{\epsilon}|^{d}+\rho^{-k}|u_{\epsilon}|^{d-1}|e^{h}-1|\bigg).
					\end{equation*}
				Notice that when $M\to \infty$ the last term in \ref{S_2_Ineqn_Int_1} vanishes, and therefore \ref{S_2_Eqn_WeightedPoincare} implies:
					{\begin{small}
					\begin{equation}\label{S_2_Ineqn_Int_2}
						\bigg(\int |\zeta u_{\epsilon}|^{ds}\rho^{s(\beta-2)-\beta}\bigg)^{\frac{1}{s}}<\frac{Cd^{2}}{d-1}\bigg(\int |e^{h}-1|\zeta |\zeta u_{\epsilon}|^{d-1}+\frac{|m|(|m|+1)d}{d-1}\int \rho^{-2}|\zeta u_{\epsilon}|^{d}\bigg).
					\end{equation}
					\end{small}}
				Let $m=0, s=1$ and use H\"older inequality, as long as $d>\frac{\beta-2+\tilde{\epsilon}}{\mu-2}$, $\tilde{\epsilon}>0$
					\begin{equation*}
						\bigg(\int \rho^{-2}|u_{\epsilon}|^{d}\bigg)^{\frac{1}{d}}<\frac{C(C_{0},\tilde{\epsilon})d^{2}}{d-1}.
					\end{equation*}
				Fix $d_{0}>\frac{\beta-2+\tilde{\epsilon}}{\mu-2}$ and $1<s<\frac{2n+2}{2n}$, define $d_{k}=s^{k}d_{0}, \zeta_{k}=\rho^{s^{k}(\beta-2)-\beta}$, then \ref{S_2_Ineqn_Int_2} reads
					\begin{equation*}
						\bigg(\int \zeta_{k+1}|u_{\epsilon}|^{d_{k+1}}\bigg)^{\frac{1}{d_{k+1}}}<\bigg(\frac{C(C_{0},\tilde{\epsilon})d_{k}^{3}}{(d_{k}-1)^{2}}\bigg)^{\frac{1}{d_{k}}}\max\bigg\{1, \bigg(\int \zeta_{k}|u_{\epsilon}|^{d_{k}}\bigg)^{\frac{1}{d_{k}}}\bigg\}.
					\end{equation*}
				And by definition $\int \zeta_{0}|u_{\epsilon}|^{d_{0}}=\int \rho^{-2}|u_{\epsilon}|^{d_{0}}<\infty$. Iteration implies
					\begin{equation*}
						|u_{\epsilon}|<C(\tilde{\epsilon},C_{0})\rho^{(\beta-2)/d_{0}}<C(\epsilon',C_{0})\rho^{2-\mu+\epsilon'}
					\end{equation*}
				Let $\epsilon\to 0$ we prove the lemma.
			\end{proof}
	\subsection{Harmonic analysis on the Calabi-Yau metric}\label{S_2_4}
		We need to study the mapping properties for the Laplacian on the double weighted space: $C^{2,\alpha}_{\delta,\tau}(X,\omega_{\phi})$.
			
			\begin{prop}\label{MAP_LAPLACIAN}[c.f. \cite{li2019gluing}, Proposition 3.3]
				Let $\tau\in(2-2n,0)$ as above and $\delta>-2n$ and avoid the indicial roots, then the Laplacian is surjective
					\begin{equation*}
						\Delta_{g_{\phi}}: C^{2,\alpha}_{\delta,\tau}\to C^{0,\alpha}_{\delta-2,\tau-2}
					\end{equation*}
				and there exists a bounded right inverse $P_{X}$.
			\end{prop}
			\begin{proof}
				Similarly as in \ref{S_2_Prop_Map_4} for the region $\{\rho>A\}$, the Laplacian $\Delta_{g_{\phi}}$ admits a bounded right inverse outside a large compact set. Therefore it suffices to invert the function $f\in C^{0,\alpha}(X)$ supported in $\{\rho<2A\}$, which implies that $|f|<C\rho^{-\mu}$ for $2<\mu<2(n+1)$ automatically. We shall apply Theorem 1.6 in \cite{hein2011weighted} or Proposition 3.16 in \cite{hein2010gravitational}. Now our manifold $X$ is of $SOB(2(n+1))$, and the weight function $\rho$ satisfies
					\begin{equation*}
						|\nabla \rho|+|\rho\Delta \rho|<C'.
					\end{equation*}
				Then there is a solution in $C^{2,\alpha}$ to the Poisson equation
					\begin{equation*}
						\Delta u=f.
					\end{equation*}
				Moreover, $u$ satisfies $|u|<C_{\epsilon}\rho^{2-\mu+\epsilon}$ for some $\epsilon\ll 1$ and some constant $C_{\epsilon}$. Since the weight function $w\leq 1$
					\begin{equation*}
						|u|<C_{\epsilon}\rho^{\delta}w^{\tau}
					\end{equation*}
				Moreover $u$ is harmonic on $\{\rho>A\}$ and then we shall use elliptic bootstrap to get the $C^{2,\alpha}_{\delta,\tau}(X)$ estimate for $u$. 
			\end{proof}
			\begin{proof}[Proof of Theorem \ref{Thm_1_Detail}]
				Use Lemma \ref{S_2_Lemma_MADecay}, we obtain a solution to
					\begin{equation*}
						(\II \p \pb (\phi''+u))^{n+1}=\II^{(n+1)^{2}}\Omega_{X}\wedge\overline{\Omega_{X}},
					\end{equation*}
				with $|u|<C(\epsilon',C_{0})\rho^{-2n+\epsilon'}<C(\epsilon',C_{0})\rho^{-2}$. Use Evans-Krylov's regularity theory on each model geometries, we are able to derive the $C^{2,\alpha}_{\delta,\tau}(X)$ bound for $u$. Meanwhile, from the our construction
					\begin{equation*}
						\phi''-\phi'\in C^{2,\alpha}_{\delta,\tau}(X).
					\end{equation*}
				we conclude that the difference between the Calabi-Yau metric and the approximate solution $\phi-\phi'\in C^{2,\alpha}_{\delta,\tau}(X)$.\\
				
				To show the uniqueness, let $u \in C^{2,\alpha}_{\delta',\tau}(X)$, $\delta'<\frac{2q}{p}$, and without loss of generality, assume $\delta'$ avoid the indicial roots. The argument is a generalization of Theorem 3.1 in \cite{conlon2013asymptotically} to warped QAC Calabi-Yau case. By the assumption
					\begin{equation*}
						(\omega_{\phi}+\II \p \pb u)^{n+1}=\omega_{\phi}^{n+1}.
					\end{equation*}
				Therefore $(\Delta_{\phi} u)\omega_{\phi}^{n+1}=(\II \p \pb u)^{2}\wedge T$, where $T$ and its derivatives are bounded. We apply Theorem 1.2 in \cite{chiu2024subquadratic}, which shows that on a non-compact Calabi-Yau manifold with maximal volume growth, any harmonic function admitting subquadratic growth is pluriharmonic. Together with the surjectivity of the Laplacian in Proposition \ref{MAP_LAPLACIAN} and iteration argument, we conclude that $u=u'+u''$ where $u'$ is pluriharmonic and $u''$ decays to $0$ near the infinity. Strong maximum principle then implies that $u''=0$.
			\end{proof}

\section{Metric ansatz for the collapsing problem}\label{S_3}
	For each $t$, Theorem \ref{Thm_1_Detail} implies that there is a unique Calabi-Yau metric on $X_{t}'$
		\begin{equation}\label{S_3_Defn_X'}
			X_{t}': P(\tilde{z})+\prod_{i=1}^{d}\prod_{j=1}^{q_{i}}(z_{n+2}-s_{i}-a_{t}^{-\theta_{i}}s_{i,j})=0,
		\end{equation}
	approximated by
		\begin{equation*}
			\pi^{*}\omega_{\CC}+t^{\frac{2(q-p)}{p}}\omega_{SRF}=\II \p \pb \big(|z_{n+2}|^{2}+t^{\frac{2(q-p)}{p}}|P(\tilde{z}')|^{2/p}\varphi(P(\tilde{z}')^{-1/p}\cdot \tilde{z}')\big),
		\end{equation*}
	outside a compact set. However, little is known about the behavior of the metric near the singular fibers, since the Calabi-Yau metric is obtained by solving a Monge Amp\`ere equation. On the other hand, gluing construction allow us to gives a favorable description of the metric, such that the gluing error is small before perturbation. Moreover the uniqueness implies that those two Calabi-Yau metrics are the same.\\
	
	As mentioned before, one of the issue is to describe the geometry of $X_{t}'$ near the infinity as the collapsing parameter $t \to \infty$. The model in \ref{S_3_Defn_X'} is not convenient for this task, for the function $\sqrt{R_{\tilde{z}'}^{2}+|z_{n+2}|^{2}}$ is not uniformly equivalent to the distance function as $t\to \infty$. To handle this issue, consider a coordinate change
		\begin{equation*}
			\tilde{z}=a_{t}^{(p-q)/p}\cdot \tilde{z}'.
		\end{equation*}
	In this coordinate the hypersurface becomes
		\begin{equation*}
			X_{t}: P(\tilde{z})+a_{t}^{q-p}Q_{t}(z_{n+2})=0,
		\end{equation*}
	where the polynomial $Q_{t}(z_{n+2})=\prod_{i=1}^{d}\prod_{j=1}^{q_{i}}(z_{n+2}-s_{i}-a_{t}^{-\theta_{i}}s_{i,j})$. And the K\"ahler potential near the infinity can be rewritten as
		\begin{equation*}
			\widetilde{\phi'}=|z_{n+2}|^{2}+|P(\tilde{z})|^{2/p}\varphi(P(\tilde{z})^{-1/p}\cdot \tilde{z}).
		\end{equation*}
	The function $\sqrt{R_{\tilde{z}}^{2}+|z_{n+2}|^{2}}$ is uniformly equivalent to the distance function. In addition to this, as $t\to \infty$, $\widetilde{\phi'}$ is a K\"ahler potential outside small neighborhoods of the singularities, and is very closed to be Calabi-Yau.\\
	
	Another issue is the rate of the singular locus converging to $s_{i}$ in the base space, such that after blow-up the limits are $X_{i}$
		\begin{equation*}
			P(\tilde{v})+Q_{i}(v_{n+2})=0,
		\end{equation*}
	where $Q_{i}=\prod_{i\neq l}\prod_{j=1}^{q_{i}}(s_{i}-s_{l})^{q_{l}}(v_{n+2}-s_{i,j})$. In a small neighborhood of $(0,s_{i})$, $Q_{t}$ is approximately $Q_{i}$. Consider the coordinate change
		\begin{equation*}
			(\tilde{z},z_{n+2}-s_{i})=(a_{t}^{x}\cdot \tilde{v},a_{t}v_{n+2}),
		\end{equation*}
	such that the defining equation together with $\widetilde{\phi'}$ in the scaled coordinate takes the form of
		\begin{equation*}
			\begin{split}
				&X_{i}: P(\tilde{v})+Q_{i}(v_{n+2})=0,\\
				&\widetilde{\phi'}\approx t^{2x}(|v_{n+2}|^{2}+|P(\tilde{v})|^{2/p}\varphi(P(\tilde{v})^{-1/p}\cdot \tilde{v}))\approx t^{2x}\phi_{i}'.
			\end{split}
		\end{equation*}
	Computation shows that $x=-\theta_{i}=-\frac{p-q}{p-q_{i}}$ is the only solution. 	Let $(X_{i},\phi_{i})$ be the Calabi-Yau manifold with the K\"ahler potential $\phi_{i}$ approximated by $\phi_{i}'$. The above discussion suggests that we should glue $(X_{i}, t^{-2\theta_{i}}\phi_{i}), 1\leq i\leq d$ to small neighborhood near $S_{i}=(0,s_{i})$ respectively. Use implicit function theorem, we obtain biholomorphism $F_{i,t}$ from large open sets on $X_{i}$ to some small neighborhoods of $S_{i}$ in $X_{t}$. In addition, the holomorphic volume form:
		\begin{equation}\label{S_3_Defn_HolErr}
			\begin{split}
				t^{-(n+1)\theta_{i}}F_{i,t}^{*}\Omega_{X_{i}}&=(1+O'(|z_{n+2}-s_{i}|))\Omega_{X_{t}}.
			\end{split}
		\end{equation}
	Here $O'(|z_{n+2}|)$ represents for some holomorphic functions with all derivatives bounded by $|z_{n+2}|$.

	\subsection{Metric Ansatz and weighted H\"older space}\label{S_3_1}
		In order to set up the gluing construction, we now define the approximate solution. Recall that
			\begin{equation*}
				\gamma_{2}(s)=
					\begin{cases}
						1 & s<1 \\
						0 & s>2
					\end{cases}.
			\end{equation*}
		Define the K\"ahler potential as,
			\begin{equation}\label{S_3_Eqn_ApproxSol}
				\begin{split}
					\phi_{t}':=&\bigg(1-\sum_{i=1}^{d}\gamma_{2}(\rho_{i}t^{-\sigma_{i}})\bigg)\bigg(|z_{n+2}|^{2}+|P(\tilde{z})|^{2/p}\varphi(P(\tilde{z})^{-1/p}\cdot \tilde{z})\bigg)\\
					+&\sum_{i}^{d}\gamma_{2}(\rho_{i}t^{-\sigma_{i}})F_{i,t}^{*}(t^{-2\theta_{i}}\phi_{i}),	
				\end{split}
			\end{equation}	
		where $(q-p)/p<\sigma_{i}<0$ and $\rho_{i}^{2}=|Q_{t}|^{2/q_{i}}+R_{\tilde{z}}^{2}$. We will show that $\phi_{t}'$ is a K\"ahler potential by estimating the corresponding bilinear form. Before introducing the weighted H\"older space, we now describe the geometry of $(X_{t},\II \p \pb \phi_{t}')$. Roughly speaking $X_{t}$ can be divided into four regions:
			\subsubsection*{Region I}
				Let $A_{1}$ be a sufficiently large constant and the region I, $U_{1}:=\{R_{\tilde{z}}>A_{1}\}\cup \{|z_{n+2}|>A_{1}\}$. The geometry in this region is similar to the geometry near the infinity as in Section \ref{S_2}, the only difference is that near the singular ray the metric $g_{\phi_{t}'}$ is locally modeled on $g_{\CC}\oplus t^{2(q-p)/p}|z_{n+2}|^{2q/p}g_{V_{1}}$.
			\subsubsection*{Region II}
				$\rho_{i}=\sqrt{|Q_{t}|^{2/q_{i}}+R_{\tilde{z}}^{2}}$ are the distance functions to $S_{i}$ in small neighborhoods respectively, we define $U_{2,i}:=\{\rho_{i}<2t^{\sigma_{i}'}\}$, $(q-p)/p<\sigma_{i}<\sigma_{i}'<0$. We glue in $(X_{i},t^{-2\theta_{i}}\phi_{i})$ to $U_{2,i}$. The image of the gluing map $F_{i,t}(U_{2,i})$ are roughly the balls of radius $2t^{\theta_{i}+\sigma_{i}'}$ in the model spaces $(X_{i},\phi_{i})$, which will exhaust $X_{i}$ as $t\to \infty$. Define $\rho_{i}'=t^{\theta_{i}}\rho_{i}$ be the distance function for $g_{\phi_{i}}$. In this region the error to the Ricci potential falls into two category: the first one is due the the difference between $\phi_{i}'$ and $\phi_{i}$, which decays as $\rho_{i}'$ approach to $\infty$, while the second one is due to the biholomorphism $F_{i,t}$ in \ref{S_3_Defn_HolErr}, which grow at the rate of $\rho_{i}'$. Our treatment to this phenomenon here is different from the one in \cite{li2019gluing}, namely the gluing regions are relatively smaller, and in some cases, for example $q_{1}=\cdots =q_{d}$ we can choose $\sigma_{i}<-\frac{1}{2}\theta_{i}$, and we can obtain a weighted $C^{1,\alpha}$ estimate for the metric while keeping the Ricci potential small in the weighted H\"older space. In general we only have weighted $C^{0,\alpha}$ estimate for the metric, but we can still bound the difference of the K\"ahler Laplacian. Although the above construction comes at the price that we will have a slightly coarser estimate in region III, it doesn't affect the gluing construction essentially.
			\subsubsection*{Region III}
				In this region $U_{3}:=\{R_{\tilde{z}}<2t^{(q-p)/p}\Lambda_{2}, |z_{n+2}|<2A_{1}, \rho_{i}>t^{\sigma_{i}'}\}$, where $\Lambda_{2}\gg A_{1}$. The metric behaves like the collapsing Calabi-Yau metric on a compact Calabi-Yau manifold fibered over $\mathbb{CP}_{1}$ \cite{li2019gluing}. The condition $\sigma_{i}'>\frac{q-p}{p}$ guarantee that $\{|z_{n+2}|<2A_{1}, \rho_{i}>t^{\sigma_{i}'}\}\subset \CC$ can be covered by $N(t)\sim t^{\frac{p-q}{p}}$ many balls of radius $2B_{2}t^{(q-p)/p}$. The metric on $\{R_{\tilde{z}}<2t^{(q-p)/p}\Lambda_{2}, |z_{n+2}-v_{0}|<2B_{2}t^{(q-p)/p}\}$, can be approximated by a semi-Ricci-flat metric $t^{2(q-p)/p}g_{\CC\times V_{Q_{t}(v_{0})}}$, here $B_{2}$ is a large constant depends on $\Lambda_{2}$.

			\subsubsection*{Region IV}
				This region $U_{4}:=\{t^{(q-p)/p}\Lambda_{2}<R_{\tilde{z}}<2A_{1}, |z_{n+2}|<2A_{1}, \rho_{i}>t^{\sigma_{i}'}\}$ is the transition region between $U_{1}$ with $U_{2}, U_{3}$, and the metric is modeled on $\CC\times V_{0}$, whose error depends only on $\Lambda_{2}, A_{1}$.\\
			
		With the understanding of the geometry of $(X_{t},\phi_{t}')$ above, we are now able to define our H\"older space for the gluing problem.
			\begin{definition}\label{S_3_Defn_Holder}
				Denote $\bs{\delta}$ as the set of index $(\delta_{0},\delta_{1},\cdots,\delta_{d})$, and $\bs{\delta-m}$ as $(\delta_{0}-m,\delta_{1}-m,\cdots ,\delta_{d}-m)$. Let $f$ be a function, define the weighted H\"older space $C^{k,\alpha}_{\bs{\delta},\tau,t}$ complying with the geometry on region I-IV as follow
					\begin{equation*}
						\begin{split}
							\|f\|_{C^{k,\alpha}_{\bs{\delta},\tau,t}}:=&A^{\delta_{0}-\tau}\|f\|_{C^{k,\alpha}_{\delta_{0},\tau,t}(U_{1})}+\sum_{i=1}^{d}t^{(\delta_{i}-\tau)\sigma_{i}'}\|f\|_{C^{k,\alpha}_{\delta_{i},\tau}	(U_{2,i})}\\
							&+\|f\|_{C^{k,\alpha}_{\tau,t}(U_{3})}+\|f\|_{C^{k,\alpha}_{\tau}(U_{4})}.
						\end{split}
					\end{equation*}
			\end{definition}
		We shall define each terms individually. In region I, we define the weight functions similarly as in \ref{S_2_Defn_W}. Let $\rho_{0}^{2}=|z_{n+2}|^{2}+R_{\tilde{z}}^{2}$ and
			\begin{equation}\label{S_3_Defn_W0}
				w_{0}=
					\begin{cases}
						1 & R>2\kappa \rho_{0}\\
						\frac{R_{\tilde{z}}}{\kappa \rho_{0}} & \kappa\rho_{0}>R_{\tilde{z}}>2t^{(q-p)/p}\kappa^{-1}\rho_{0}^{q/p}\\
						t^{(q-p)/p}\kappa^{-2}\rho_{0}^{(q-p)/p} & R_{\tilde{z}}<t^{(q-p)/p}\kappa^{-1}\rho_{0}^{q/p}
					\end{cases}.
			\end{equation}
		The rest of the definition for the double weighted space $C^{k,\alpha}_{\delta_{0},\tau,t}(U_{1})$ is the same as in \ref{S_2_Defn_Holder_1}, \ref{S_2_Defn_Holder_2}. Here, we should point out that $t$-dependent weights are design to handle the issue that a large sphere in $X_{t}$ will develop a circle of singularities as $t \to \infty$. Meanwhile it recovers a distance function on $V_{1}$. Recall that $\xi$ is a function on $V_{1}: P(\tilde{z})+1=0$, and is a smoothing for $\max\{\kappa^{-1}, R_{\tilde{z}}\}$ and we define
			\begin{equation*}
				\xi_{a}:=|a|^{1/p}S_{a}^{*}\xi,
			\end{equation*}
		where the map $S_{a}$ is given by
			\begin{equation*}
				S_{a}: V_{a}\to V_{1}, \tilde{z}\mapsto a^{-1/p}\cdot \tilde{z}.
			\end{equation*}The function $\xi_{a}$ is a distance function on $P(\tilde{z})+a=0$, which is uniformly equivalent to $R_{\tilde{z}}$ when $\xi_{a}$ is large.
		Now by definition $\rho_{0}w_{0}$ is uniformly equivalent to $\xi_{a_{t}^{q-p}Q_{t}}$, here we treat $\kappa$ as a constant.\\
		
		On region II, we the use the scaled metric $t^{2\theta_{i}}g_{\phi_{t}'}$ to define the weighted H\"older space on $U_{2,i}$. Here $(\tilde{v},v_{n+2})$ for the scaled coordinate in $\CC^{n+2}$. Recall that $\rho_{i}'=t^{\theta_{i}}\rho_{i}$
			\begin{equation}\label{S_3_Defn_WI}
				w_{i}=
					\begin{cases}
						1 & R_{\tilde{v}}> 2\kappa \rho_{i}'\\
						\frac{R_{\tilde{v}}}{\kappa \rho_{i}'} & \kappa \rho_{1}'>R_{\tilde{v}}>2\kappa^{-1}(\rho_{i}')^{q_{i}/p}\\
						\kappa^{-2}(\rho_{i}')^{q_{i}/p-1} & R_{\tilde{v}}<\kappa^{-1}(\rho_{1}')^{q_{i}/p}
					\end{cases}.
			\end{equation}
		Notice that the norm $\|\cdot\|_{C^{k,\alpha}_{0,0}}(U_{2,i})$ is invariant under scaling, therefore we now define
			\begin{equation*}
				t^{(\delta_{i}-\tau)\sigma_{i}'}\|f\|_{C^{k,\alpha}_{\delta_{i},\tau}(U_{2,i})}:= t^{\delta_{i}(\sigma_{i}'+\theta_{i})-\tau\sigma_{i}'}\|f\|_{C^{k,\alpha}_{\delta_{i},\tau}(F_{i}(U_{2,i}),t^{2\theta_{i}}g_{\phi_{t}'})},
			\end{equation*}
		where $\|\cdot \|_{C^{k,\alpha}_{\delta_{i},\tau}(F_{i}(U_{2,i}),t^{2\theta_{i}}g_{\phi_{t}'})}$ are defined similarly as in the double weighted H\"older space on $(X_{i}, g_{\phi_{i}'})$. It should be remarked that those $t$ factors are designed to counteract the effect introduced by $(\rho_{i}')^{-\delta_{i}}w_{i}^{-\tau}$ under the scaling.\\ 
		
		In the last step, define
			\begin{equation*}
				\|f\|_{C^{k,\alpha}_{\tau,t}(U_{3})}:=\|\xi_{a_{t}^{q-p}Q_{t}}^{-\tau}f\|_{C^{k,\alpha}(\xi_{a_{t}^{q-p}Q_{t}}^{-2}g_{\phi_{t}})},
			\end{equation*}
		and
			\begin{equation*}
				\|f\|_{C^{k,\alpha}_{\tau}(U_{4})}:=\|R_{\tilde{z}}^{-\tau}f\|_{C^{k,\alpha}(R_{\tilde{z}}^{-2}g_{\phi_{t}})},
			\end{equation*}
		as the weighted H\"older norm on $U_{3},U_{4}$. Our weight functions are chosen so that they are uniformly equivalent on the overlap, in the sense that the bound are independent on the collapsing parameter $t$. It should be remarked that the H\"older norm $C^{k,\alpha}_{\bs{0},0,t}$ is invariant under scaling. Let $T$ be a tensor of the type $(r,s)$
			\begin{equation*}
				T\in \Gamma\big(\bigotimes^{r}TX\otimes \bigotimes^{s}T^{*}X \big)
			\end{equation*}
		we can define the H\"older norm similarly, but we need to consider the weighted H\"older norm for $t^{(s-r)\theta_{i}}T$ under the scaled metric in region II.
		
	\subsection{Metric approximation and Ricci potential}\label{S_3_2}
		In the previous section we set up the weighted H\"older space for the gluing construction. However for some technical issue, we are only able to show that the Ricci potential $h_{t}$ for the K\"ahler metric $(X_{t},\omega_{\phi_{t}'})$ is small given that $p,q_{i}$ satisfy a numerical condition. The Proposition \ref{S_2_Prop_Ricci_1} on the decay rate of the Ricci potential to the approximate solution on $X_{i}$ is crucial in our gluing construction, in the sense that we rely on the optimal estimate for the K\"ahler potential.
			\begin{prop}\label{S_3_Prop_Ricci}
					Suppose $p>2q_{i}$ and if $q_{1}<\frac{1}{4}p<q_{d}$, suppose further that $\frac{p+4q_{1}}{5p-4q_{1}}>\frac{2q_{d}}{3p-2q_{d}}$. Take $\alpha,\tau, \sigma_{i}'-\sigma_{i}$ close to $0$ and $\delta_{_{i}},1\leq i\leq d$ sufficiently close to $\frac{2q_{i}}{p}$, then there exists a positive constant $a$ such that the Ricci potential satisfies
						\begin{equation*}
							\|h_{t}\|_{C^{0,\alpha}_{\bs{\delta-2},\tau-2,t}}\lesssim t^{-(2-\tau)(p-q)/p-a}.
						\end{equation*}
					Moreover if a function $f$ satisfies
						\begin{equation*}
							\|f\|_{C^{0,\alpha}_{\bs{\delta-2},\tau-2,t}}\lesssim t^{-(2-\tau)(p-q)/p-a},
						\end{equation*}
					then $f$ is small in the scaled models in region I-IV as $t\to \infty$. To be specific:
						\begin{equation*}
							\|f\|_{C^{0,\alpha}_{\bs{0},0,t}}\lesssim \epsilon(t),
						\end{equation*}
					where $\epsilon(t)$ denotes the terms bounded by $t^{-\epsilon}$ for some positive $\epsilon$.
			\end{prop}

		In region I the Ricci potential $h_{t}$ satisfies:
			\begin{lemma}\label{S_3_Lemma_Ricci_1}
				Let $h_{t}$ be the Ricci potential to $\omega_{\phi_{t}'}$
					\begin{equation*}
						|\nabla_{g_{\phi_{t}'}}^{k}h_{t}|<
							\begin{cases}
								C_{k}t^{2(q-p)}\rho_{0}^{2(q-p)-k}& R_{\tilde{z}}>\kappa \rho_{0}\\
								C_{k}t^{2(q-p)}\rho_{0}^{2(q-1)}R_{\tilde{z}}^{-2(p-1)-k} & \kappa\rho_{0}>R_{\tilde{z}}>\kappa^{-1}t^{(q-p)/p}\rho_{0}^{q/p}\\
								C_{k}(\kappa)(t\rho_{0})^{2(q-p)/p}(t^{q-p}\rho_{0}^{q})^{-k/p}  & R_{\tilde{z}}<\kappa^{-1}t^{(q-p)/p}\rho_{0}^{q/p}
							\end{cases}.
					\end{equation*}
				In the language of the weighted H\"older space is
					\begin{equation*}
						\|h_{t}\|_{C^{k,\alpha}_{\delta_{0},\tau}(U_{1})}\lesssim t^{(4-\tau)(q-p)/p}.
					\end{equation*}
			\end{lemma}
			\begin{proof}
				Notice that the Ricci potential is dominated by
					\begin{equation*}
						h_{t}=t^{2(q-p)}\frac{(\II \p \pb (|P(\tilde{z})|^{2/p}\varphi(P(\tilde{z})^{-1/p}\cdot \tilde{z})))^{n+1}\wedge \II \ud Q_{t}\wedge \ud \overline{Q_{t}}}{\II \ud z_{1}\wedge \cdots \ud \overline{z_{n+2}}}.
					\end{equation*}
				The computation in the $R_{\tilde{z}}>\kappa^{-1}t^{(q-p)/p}\rho_{0}^{q/p}$ is the same as in \ref{S_2_Prop_Ricci_1}. The computation on $R_{\tilde{z}}<\kappa^{-1}t^{(q-p)/p}\rho_{0}^{q/p}$ is slightly different. Consider the coordinate change
					\begin{equation}\label{S_3_Eqn_Coordinate_1}
						Q_{t}(z_{n+2})=v_{n+2}^{p}, \tilde{z}=t^{(q-p)/p}v_{n+2}^{q/p}\cdot \tilde{v}.
					\end{equation}
				The defining equation is
					\begin{equation*}
						P(\tilde{v})+1=0, R_{\tilde{v}}\lesssim \kappa^{-1}.
					\end{equation*}
				In the region $\{|v_{n+2}-v_{0}|<B_{1}t^{(q-p)/p}|v_{0}|^{q/p}\}$, use $|v_{n+2}|^{2}+|v_{0}|^{2q/p}\varphi(\tilde{v})$ to approximate the K\"ahler potential. We therefore conclude that
					\begin{equation*}
						|\nabla_{g_{\phi_{t}'}}^{k} h_{t}|<C_{k}(\kappa,B_{1})(t\rho_{0})^{2(q-p)/p}(t^{q-p}\rho_{0}^{q})^{-k/p}.
					\end{equation*}
			\end{proof}
		We shall now consider the comparison between $g_{\phi_{t}'}$ and $g_{\CC\times V_{0}}$ in a slightly larger domain.
			\begin{lemma}\label{S_3_Lemma_Met_1}
				Fix $k \in \N$, $\forall \epsilon>0$, there exist $A>A_{1}(\epsilon,B_{1}), B_{1}>B_{1}(\epsilon,\Lambda_{1}), \Lambda_{1}>\Lambda_{1}(\epsilon)$, then on $U_{1,1}':=\{R_{\tilde{z}}>\Lambda_{1}^{3/4}t^{(q-p)/p}\rho_{0}^{q/p}\}\cap U_{1}'$, where $U_{1}'=\{|z_{n+2}|, R_{\tilde{z}}>A_{1}^{3/4}\}$, the metric is approximated by $g_{\CC\times V_{0}}$
					\begin{equation*}
						\|\nabla_{g_{\phi_{t}'}}^{k}(g_{\phi_{t}'}-g_{\CC\times V_{0}})\|_{C^{k,\alpha}_{0,0}(U_{1,1}')}<\epsilon,
					\end{equation*}
				and on $U_{1,i}=\{R_{\tilde{z}}<2\Lambda_{1}^{2}t^{(q-p)/p}\rho_{0}^{q/p}, |z_{n+2}-z_{0}^{i}|<B_{1}t^{(q-p)/p}|z_{0}^{i}|^{q/p}\}\cap X_{t}, i>1$ we can approximate $g_{\phi_{t}'}$ by $t^{2(q-p)/p}|z_{0}^{i}|^{2q/p}g_{\CC\times V_{0}}$
					\begin{equation*}
						|\nabla_{g_{\phi_{t}'}}^{k}(g_{\phi_{t}'}-t^{2(q-p)/p}|z_{0}^{i}|^{2q/p}g_{\CC\times V_{0}})|<\epsilon \xi_{a_{t}^{q-p}Q_{t}}^{-k}.
					\end{equation*}
			\end{lemma}
			\begin{proof}
				The proof very similar to Lemma \ref{S_3_Lemma_Ricci_1}, we can again divide $X_{t}$ into three parts.
					\subsubsection*{$R_{\tilde{z}}>\kappa \rho_{0}$:}
						Consider the coordinate change
							\begin{equation*}
								\tilde{z}=D\cdot \tilde{v}, z_{n+2}=Dv_{n+2}, D<\rho_{0}<2D.
							\end{equation*}
						The definition equation becomes
							\begin{equation*}
								P(\tilde{v})+a_{t}^{q-p}D^{q-p}Q_{D}(v_{n+2})=0.
							\end{equation*}
						The difference between $D^{-2}g_{\phi_{t}'}$ and $D^{-2}g_{\CC\times V_{0}}$ is dominated by two parts. The first part is due to the projection to $V_{0}$ in the $\CC^{n+1}$ factor and all of its $k$-th derivatives bounded by $C_{k}(tD)^{q-p}$. And the seconded part is bounded by $D^{-2}Err$, whose $k$-th derivatives are bounded by $C_{k}(tD)^{-(2+c)(p-q)/p}$.
					
					\subsubsection*{$2\kappa \rho_{0}>R_{\tilde{z}}>\Lambda_{1}^{3/4}t^{(q-p)/p}\rho_{0}^{q/p}$:}
						Consider the coordinate change
							\begin{equation*}
								\tilde{z}=K\cdot \tilde{w}, z_{n+2}=v_{0}+K v_{n+2}.
							\end{equation*}
						Where $D<\rho_{0}<2D, K<R_{\tilde{z}}<2K$, and $\Lambda_{1}^{3/4}t^{(q-p)/p}<K<2\rho_{0}$. Similar argument applies. The term due to projection and its $k$-th derivatives are bounded by $C_{k}t^{q-p}D^{q-1}K^{-(p-1)}$. Meanwhile the $k$-th derivatives of the second term are bounded by $C_{k}(t^{(q-p)/p}D^{q/p}K^{-1})^{-(2+c)}$. It follows from the condition that $R_{\tilde{z}}>\Lambda_{1}^{3/4}t^{(q-p)/p}\rho_{0}^{q/p}$, the error terms are bounded by $C_{k}\Lambda_{1}^{-3(2+c)/4}$.

					\subsubsection*{$R_{\tilde{z}}<2\Lambda_{1}^{2}t^{(q-p)/p}\rho_{0}^{q/p}$:}
						We consider the region
							\begin{equation*}
								\{|z_{n}-z_{0}^{i}|<3B_{1}t^{(q-p)/p}|z_{0}^{i}|^{q/p}, R_{\tilde{z}}<2\Lambda_{1}^{2}t^{(q-p)/p}\rho_{0}^{q/p}\}.
							\end{equation*}
						
						 Use the same coordinate change in \ref{S_3_Eqn_Coordinate_1}. Now compare the scale K\"ahler potential
							\begin{equation*}
								t^{2(p-q)/p}|v_{0}^{i}|^{-2q/p}\phi_{t}'=|Q_{t}^{-1}(v_{n+2}^{p})|^{2}+\bigg|\frac{v_{n+2}}{v_{0}^{i}}\bigg|^{2q/p}\varphi(\tilde{v}),
							\end{equation*}
						with $|v_{n+2}|^{2}+\varphi(\tilde{v})$. Use the fact that $\varphi\in C^{\infty}_{2}(V_{1})$. Computation show that, the difference $Err'$ satisfies
							\begin{equation*}
								|\nabla^{k}Err'|<C_{k}B_{1}^{-k}\Lambda_{1}^{2(1+k)}\xi_{1}^{2-k}.
							\end{equation*}
						Then we prove the theorem, by observing that the difference for the metric is bounded by $\nabla^{2}Err'$.
						
			\end{proof}
		
		In the region II, the error terms in Ricci potential can be divided into two parts, the first part arises as the error introduced by the biholomorphism $F_{i,t}$, and the second part is introduced by gluing $\phi_{i}'$ to $\phi_{i}$ on $X_{i}$. We use $(\tilde{v},v_{n+2})$ as coordinates on $X_{i}\subset \CC^{n+2}$. Now $t^{2\theta_{i}}\phi_{t}'$ can be written as follow (up to a pluriharmnic function)
			\begin{equation*}
				\begin{split}
					\widetilde{\phi_{i,t}}:&=t^{2\theta_{i}}\phi_{_{t}}'=\gamma_{2}(\rho_{i}'t^{-\sigma_{i}-\theta_{i}})\phi_{i}\\
					&+\gamma_{1}(\rho_{i}'t^{-\sigma_{i}-\theta_{i}})\bigg((F_{i,t}^{-1})^{*}(t^{2\theta_{i}}|z_{_{n+2}}|^{2})+|P(\tilde{v})|^{2/d}\varphi(P(\tilde{v})^{-1/d}\cdot \tilde{v})\bigg).
				\end{split}
			\end{equation*}
		By the construction
			\begin{equation*}
				\begin{split}
					&(F_{i,t}^{-1})^{*}(t^{\theta_{i}}z_{n+2})=(1+t^{-\theta_{i}}O'(|v_{n+2}|))v_{n+2},\\
					&t^{(n+1)\theta_{i}}(F_{i,t}^{-1})^{*}\Omega_{X_{t}}=(1+t^{-\theta_{i}}O'(|v_{n+2}|))\Omega_{X_{_{i}}}.
				\end{split}
			\end{equation*}
		We can decompose $\widetilde{\phi_{i,t}}-\phi_{i}=u_{i,1}+u_{i,2}$ such that
			\begin{equation*}
				u_{i,1}=\gamma_{1}(\rho_{i}'t^{-\sigma_{i}-\theta_{i}})\bigg(\phi_{i}-(|v_{n+2}|^{2}+|P(\tilde{v})|^{2/p}\varphi(P(\tilde{v})^{-1/p}\cdot \tilde{v}))\bigg).
			\end{equation*}
		The first term $u_{1,i}$ supports on $\{\rho_{i}'>t^{\sigma_{i}+\theta_{i}}\}$ and $u_{1,i}\in C^{k,\alpha}_{\delta_{i},\tau}(X_{1})$ by Theorem \ref{Thm_1_Detail}. On the other hand, the error introduced by biholomorphism $F_{i,t}$ can be estimated
			\begin{equation*}
				|\nabla_{g_{\widetilde{\phi_{i,t}}}}^{k}u_{i,2}|<C_{k}t^{-\theta_{i}}\rho_{i}'.
			\end{equation*}
		The estimate for the metrics follows naturally, let $U_{2,i}'=\{\rho_{i}
		'<2A_{2}^{2}t^{\sigma_{i}'+\theta_{i}}\}$ be a slightly larger open set that contain $U_{2,i}$.
			\begin{lemma}\label{S_3_Lemma_Met_2}
				If $\alpha$ is sufficiently closed to $0$, then the $C^{0,\alpha}_{0,0}$ estimate for the difference is given by
					\begin{equation*}
						\|g_{\widetilde{\phi_{i,t}}}-g_{\phi_{i}}\|_{C^{0,\alpha}_{0,0}(U_{2,i}',g_{\widetilde{\phi_{i,t}}})}<\epsilon(t).
					\end{equation*}
				Recall that here $\epsilon(t)$ is some negative power of $t$.
			\end{lemma}
			\begin{proof}
				Our transition region for the gluing is $\{t^{\sigma_{i}+\theta_{i}}<\rho_{i}'<2t^{\sigma_{i}'+\theta_{i}}\}$ on $X_{i}$. The difference between the metrics is bounded by $\nabla^{2}(u_{i,1}+u_{i,2})$, which satisfy
					\begin{equation*}
						\begin{split}
							|\nabla_{g_{\widetilde{\phi_{i,t}}}}^{k+2}u_{i,1}|&<C_{k}(\rho_{i}')^{\delta_{i}-\tau}(\rho_{i}'w_{i})^{\tau-2}(\rho_{i}'w_{i})^{-k}\\
							&<C_{k}(\rho_{i}')^{\delta_{i}-2q_{i}/p-\tau(1-q_{i}/p)}(\rho_{i}'w_{i})^{-k}.
						\end{split}
					\end{equation*}
				Choose $\tau$ sufficiently closed to $0$ such that $\delta_{i}-2q_{i}/p-\tau(1-q_{i}/p)<0$, we conclude that
					\begin{equation*}
						\|u_{i,1}\|_{C^{k,\alpha}_{0,0}(\{\rho_{i}'>t^{\sigma_{i}+\theta_{i}}\})}<C_{k}\epsilon(t).
					\end{equation*}
				Meanwhile for the second term
					\begin{equation*}
						|\nabla_{g_{\widetilde{\phi_{i,t}}}}^{k+2}u_{i,2}|<C_{k}t^{-\theta_{i}}\rho_{1}',
					\end{equation*}
				implies
					\begin{equation}\label{S_3_Eqn_Met_2}
						\begin{split}
							&|(\rho_{i}')^{k}\nabla_{g_{\widetilde{\phi_{i,t}}}}^{k+2}u_{i,2}|\lesssim t^{(k+1)\sigma_{i}'+k\theta_{i}}, k\leq 2,\\
							\sup_{v\neq v', v'\in B_{c}(v)}&(\rho_{i}'w_{i})^{\alpha}\frac{|\nabla^{2}u_{i,2}(v)-\nabla^{2}u_{i,2}(v')|}{d(v,v')^{\alpha}}\lesssim t^{\sigma_{i}'(1+\alpha)+\alpha\theta_{i}}.
						\end{split}
					\end{equation}
				We conclude that
					\begin{equation*}
						\|g_{\widetilde{\phi_{i,t}}}-g_{\phi_{i}}\|_{C^{0,\alpha}_{0,0}(U_{2,i}',g_{\widetilde{\phi_{i,t}}})}<\epsilon(t).
					\end{equation*}
			\end{proof}
			\begin{remark}\label{S_3_Remark_Met_2}
				When $p>2q_{i}$, using \ref{S_3_Eqn_Met_2} we can have $C^{1,\alpha}_{0,0}$ estimate by setting $(q-p)/p<\sigma_{i}'<-\theta_{i}/2$. Furthermore if $q_{1}=q_{2}$ we can choose $\sigma_{1}'=\sigma_{2}'$, and we can choose $\sigma_{i}'<-\frac{3}{4}\theta_{i}$ to obtain a weighted $C^{2,\alpha}$ estimate for the metric, while the conclusion in Proposition \ref{S_3_Prop_Ricci} still holds. It is useful for controlling the difference in the Laplacians between the approximate solution and the model spaces when gluing these model spaces to non-K\"ahler Einstein manifolds.
			\end{remark}
		
		Using similar argument, on $U_{2,i}$ the Ricci potential decomposes into two parts $h_{t}=h_{i,1}+h_{i,2}$ such that $h_{i,1}\in C^{k,\alpha}_{\delta_{i}'-2,\tau-2}(X_{i}),\delta_{i}'<\delta_{i}$, and supported on the set $\{\rho_{i}'>t^{\sigma_{i}+\theta_{i}}\}$; while $h_{i,2}$ satisfies
			\begin{equation*}
				|\nabla_{g_{\widetilde{\phi_{i,t}}}}^{k}h_{i,2}|<C_{k}t^{-\theta_{i}}\rho_{i}'.
			\end{equation*}
			\begin{lemma}\label{S_3_Lemma_Ricci_2}
				The Ricci potential satisfies
					\begin{equation}\label{S_3_Eqn_Ricci_2}
						t^{(\sigma_{i}'+\theta_{i})(\delta_{i}-\tau)}\|h_{t}\|_{C^{0,\alpha}_{\delta_{i}-2,\tau-2}(U_{2,i})}\lesssim t^{b},
					\end{equation}
				where $b=\max\{\delta_{i}(\sigma_{i}'-\sigma_{1})+\delta_{i}'(\sigma_{i}+\theta_{i})-\tau\sigma_{i}'-2\theta_{i},(3+\alpha-\tau)\sigma_{i}'+\alpha\theta_{i}\}$. We postpone the computation of $b$ until the proof of Proposition \ref{S_3_Prop_Ricci}.
			\end{lemma}
			\begin{proof}
				The strategy is to transfer the computation to the scaled metric $(X_{i},\widetilde{\phi_{i,t}})$, and then directly compute the H\"older norms of $h_{i,1}, h_{i,2}$ separately. We start with $h_{i,1}$. The construction of  $h_{i,1}\in C^{k,\alpha}_{\delta_{i}',\tau}(X_{i})$ and the fact that it supports on $\{\rho_{i}'>t^{\sigma_{i}+\theta_{i}}\}$ imply
					\begin{equation*}
						|(\rho_{i}')^{2+k-\delta_{i}}w_{i}^{2+k-\tau}\nabla_{g_{\widetilde{\phi_{i,t}}}}^{k}h_{i,1}|<C_{k}(\rho_{i}')^{\delta_{i}'-\delta_{i}}
					\end{equation*}
				Proposition \ref{S_2_Prop_Ricci_1} implies we can choose $2(2q_{i}-p)/p<\delta_{i}'<\delta_{i}<2q_{i}/p$. Now compute the weighted H\"older norm for $h_{i,1}, h_{i,2}$ directly by definition
					\begin{equation*}
						\begin{split}
							&t^{\delta_{i}(\sigma_{i}'+\theta_{i})-\tau\sigma_{i}'-2\theta_{i}}|(\rho_{i}')^{2-\delta_{1}}w_{i}^{2-\tau}\nabla_{g_{\widetilde{\phi_{i,t}}}}^{k}h_{i,1}|\\
							&<C_{k}t^{\delta_{i}(\sigma_{i}'-\sigma_{i})+\delta_{i}'(\sigma_{i}+\theta_{i})-\tau\sigma_{i}'-2\theta_{i}}<C_{k}t^{b}.
						\end{split}
					\end{equation*}
				Meanwhile, from $|\nabla_{g_{\widetilde{\phi_{i,t}}}}^{k}h_{i,2}|<C_{k}t^{-\theta_{i}}\rho_{i}'$, we obtain
					\begin{equation*}
						\begin{split}
							&t^{\delta_{i}(\sigma_{i}'+\theta_{i})-\tau\sigma_{i}'-2\theta_{i}}(\rho_{i}')^{2-\delta_{i}}w_{1}^{2-\tau}|h_{i,2}|\lesssim t^{3\sigma_{i}'-\tau\sigma_{i}'}<t^{b};\\
							&t^{\delta_{i}(\sigma_{i}'+\theta_{i})-\tau\sigma_{i}'-2\theta_{i}}\sup_{v\neq v', v'\in B_{c}(v)}(\rho_{1}')^{-\delta_{1}+2+\alpha}w_{1}^{-\tau+2+\alpha}
							\frac{|h_{i,2}(v)-h_{i,2}(v')|}{d(v,v')^{\alpha}}\\
							&\lesssim t^{(3+\alpha-\tau)\sigma_{i}'+\alpha\theta_{i}}\leq t^{b}.
						\end{split}
					\end{equation*}
				To sum up the discussion above, we conclude that
					\begin{equation*}
						t^{(\sigma_{i}'+\theta_{i})(\delta_{i}-\tau)}\|h_{t}\|_{C^{0,\alpha}_{\delta_{i}-2,\tau-2}(U_{2,i})}<t^{b}.
					\end{equation*}
			\end{proof}
			
		In region III, we decompose the base space into balls of radius $2B_{2}t^{(q-p)/p}$, we define
			\begin{equation*}
				U_{3,i}=\{|z_{n+2}-v_{0}^{i}|<2B_{2}t^{(q-p)/p}\}\cap U_{3}.
			\end{equation*}
		Consider the coordinate change
			\begin{equation*}
				\begin{split}
					&z_{n+2}=v_{0}^{i}+a_{t}^{(q-p)}v_{n+2},\\
					&\tilde{z}=\big(a_{t}^{(q-p)}Q_{t}(z_{n+2})/Q_{t}(v_{0}^{i})\big)^{1/p}\cdot \tilde{v}.
				\end{split}
			\end{equation*}
		Now the defining equation becomes
			\begin{equation*}
				P(\tilde{v})+Q_{t}(v_{0}^{i})=0, R_{\tilde{v}}\lesssim \Lambda_{2}, |v_{n+2}|<2B_{2}.
			\end{equation*}
		In this coordinate, up to a pluriharmonic function, the K\"ahler potential can be written as
			\begin{equation*}
				t^{2(p-q)/p}\phi_{t}'=|v_{n+2}|^{2}+\bigg|\frac{Q_{t}(v_{0}^{i}+t^{(q-p)/p}v_{n+2})}{Q_{t}(v_{0})}\bigg|^{2/p}\varphi_{i}(\tilde{v}),
			\end{equation*}
		where $\varphi_{i,t}:=|Q_{t}(v_{0}^{i})|^{2/d}\varphi(Q_{t}(v_{0}^{i})^{-1/d}\cdot \tilde{v})$ defines a Calabi-Yau metric on
			\begin{equation*}
				V_{Q_{t}(v_{0}^{i})}\subset \CC^{n+1}: P(\tilde{v})+Q_{t}(v_{0}^{i})=0.
			\end{equation*}
		For simplicity write $\phi_{i,t}'=|v_{n+2}|^{2}+\varphi_{i,t}$ and $\xi_{i}=\xi_{Q_{t}(v_{0}^{i})}$ as a distance function on $V_{Q_{t}(v_{i})}$. Our construction in region II hinders $|Q_{t}|$ to be bounded from below, which might lead to the error term in both metric and Ricci curvature failing to converge to $0$ as $t\to \infty$. However $\sigma':=\min\{\sigma_{i}'\}>(q-p)/p$ ensures our machinery is applicable in this case, roughly speaking the "bad term" is characterized by $|\ud \ln(Q_{t})|$, which is bounded by $t^{-\sigma'}$. This factor can be absorbed by $t^{(q-p)/p}$, making the error small enough as $t\to \infty$.

			\begin{lemma}\label{S_3_Lemma_Met_3}
				The difference between metric $g_{\phi_{t}'}$ and $t^{2(q-p)/p}g_{\phi_{i}}$ is bounded by
					\begin{equation*}
						\|g_{\phi_{t}}-t^{2(q-p)/p}g_{\phi_{i,t}}\|_{C^{k,\alpha}_{0}(U_{3,i})}<C_{k}(A_{1},B_{2},\Lambda_{2})t^{(q-p)/p-\sigma'}.
					\end{equation*}
			\end{lemma}
			\begin{proof}
				Notice that the H\"older $C^{k,\alpha}_{0}(U_{3})$ is invariant under scaling therefore for simplicity we only need to derive the estimate for the scaled metric $t^{2(p-q)/p}g_{\phi_{t}'}$. The key to the proof is that the K\"ahler potential $\varphi$ on $V_{1}$ satisfies
					\begin{equation*}
						\varphi\in C^{\infty}_{2}(V_{1}),
					\end{equation*}
				which implies that there are some uniform constant $C_{k}$ independent on $Q_{t}$ such that
					\begin{equation}\label{S_3_Eqn_Diff_1}
						|\nabla^{k}_{g_{\phi_{i,t}}}\varphi_{i,t}|_{g_{\phi_{i,t}}}<C_{k}\xi_{i}^{2-k}.
					\end{equation}
				Write $Q_{i}:=\big|\frac{Q_{t}(v_{0}+t^{(q-p)/p}v_{n+2})}{Q_{t}(v_{0}^{i})}\big|$ for short, combine with the fact that:
					\begin{equation}\label{S_3_Eqn_Diff_2}
						|\nabla^{k}_{g_{\phi_{i,t}}}(Q_{i}^{2/p}-1)|_{g_{\phi_{i,t}}}<C_{k}B_{2}^{k}t^{(q-p)/p-\sigma'}<C_{k}B_{2}^{k}\Lambda_{2}^{k} t^{(q-p)/p-\sigma'}\xi_{i}^{-k}.
					\end{equation}
				Notice that
					\begin{equation*}
						t^{2(p-q)/p}\phi_{t}'=|v_{n+2}|^{2}+Q_{i}^{2/p}\varphi_{i,t},
					\end{equation*}
				therefore we obtain
					\begin{equation*}
						|\nabla_{g_{\phi_{i,t}}}^{k}(g_{\phi_{i,t}}-t^{(q-p)/p}g_{\phi_{t}'})|_{g_{\phi_{i,t}}}<C_{k}(B_{2},\Lambda_{2})t^{(q-p)/p-\sigma'}\xi_{i}^{-k}.
					\end{equation*}
				Use the property that this norm is invariant under scaling, we prove the lemma.
			\end{proof}
		Similar to the computation near the singular ray in region I, the Ricci potential on region III can be bounded by
			{\begin{small}
			\begin{equation}\label{S_3_Eqn_Ricci_1}
				\begin{split}
					t^{\frac{2(q-p)}{p}}|Q_{t}'(v_{0}^{i}+t^{\frac{(q-p)}{p}}v_{n+2})|Q_{i}^{\frac{-2(n+p)}{p}}
					\frac{(\II \p \pb(|v_{n+2}|^{2}+Q_{i}^{\frac{2}{p}}\varphi_{i,t}))^{n+1}\wedge \II\ud v_{n+2}\wedge \ud \overline{v_{n+2}}}{\II \ud v_{1}\wedge \cdots \wedge \ud\overline{v_{n+2}}}.
				\end{split}
			\end{equation}
			\end{small}}
			\begin{lemma}\label{S_3_Lemma_Ricci_3}
				On region III, if $\tau$ sufficiently closed to $0$, the Ricci potential satisfies
					\begin{equation*}
						\|h_{t}\|_{C^{k,\alpha}_{\tau-2}(U_{3})}\lesssim t^{(4-\tau)(q-p)/p}.
					\end{equation*}
			\end{lemma}
			\begin{proof}
				The Equations \ref{S_3_Eqn_Diff_1}, \ref{S_3_Eqn_Diff_2}, \ref{S_3_Eqn_Ricci_1} implies
					\begin{equation*}
						|\nabla_{g_{\phi_{i,t}}}^{k}h_{t}|_{g_{\phi_{i,t}}}<C_{k}(A_{1},B_{2},\Lambda_{2})t^{2(q-p)/p}\xi_{i}^{-k}.
					\end{equation*}
				Since this expression is scaled invariant, we can transfer the computation back to our original metric
					\begin{equation*}
						\xi_{a_{t}^{q-p}Q_{t}}^{k}|\nabla_{g_{\phi_{t}}'}^{k}h_{t}|<C_{k}(A_{1},B_{2},\Lambda_{2})t^{2(q-p)/p}.
					\end{equation*}
				Now on region III, $\xi_{a_{t}^{q-p}Q_{t}}\lesssim t^{(q-p)/p}\Lambda_{2}$
					\begin{equation*}
						\xi_{a_{t}^{q-p}Q_{t}}^{k+2-\tau}|\nabla_{g_{\phi_{t}}'}^{k}h_{t}|<C_{k}(A_{1},B_{2},\Lambda_{2})t^{(4-\tau)(q-p)/p}.
					\end{equation*}
			\end{proof}
		\begin{lemma}\label{PROP_DECAY_C}
			In region IV, the metric and the Ricci potential
				\begin{equation}\label{S_3_Lemma_RicciMet_4}
					\begin{split}
						&\|g_{\phi_{t}'}-g_{\CC\times V_{0}}\|_{C^{k,\alpha}_{0}(U_{4})}<C_{k}(A_{1})\Lambda_{2}^{-2-c},\\
						&\|h\|_{C^{k,\alpha}_{\tau}(U_{4})}\lesssim t^{(4-\tau)(q-p)/p}.
					\end{split}
				\end{equation}
		\end{lemma}
		\begin{proof}
			Estimate for the metric and Ricci potential in this region is rather coarse, though they are suffice for our purpose. Consider the coordinate change, in the region $\{K<R_{\tilde{z}}<2K\}\cap\{|z_{n+2}-v_{0}|<K\}$, with $t^{(q-p)/p}\Lambda_{2}<K<A_{1}/2$ (choose $\Lambda_{2}\gg A_{1}$).
				\begin{equation*}
					\tilde{z}=K\cdot \tilde{v}, z_{n+2}=v_{0}+K v_{n+2}, 
				\end{equation*}
			The defining equation is
				\begin{equation*}
					P(\tilde{v})+a_{t}^{q-p}K^{-p}Q_{t}(v_{0}+Kv_{n+2})=0.
				\end{equation*}
			Observe that
				\begin{equation*}
					|a_{t}^{q-p}K^{-p}Q_{t}|\leq Ct^{q-p}K^{-p}A_{1}^{q}\leq CA_{1}^{q}\Lambda_{2}^{-p}.
				\end{equation*}
			Now
				\begin{equation*}
					K^{-2}Err:=|P(\tilde{v})|^{2/p}\varphi(P(\tilde{v})^{-1/p}\cdot \tilde{v})-r_{\tilde{v}}^{2},
				\end{equation*}
			and therefore
				\begin{equation*}
					|\nabla^{k}_{K^{-2}g_{\CC\times V_{0}}}K^{-2}Err|_{K^{-2}g_{\CC\times V_{0}}}<C_{k}(t^{q-p}K^{-p}A_{1}^{q})^{(2+c)/p}	A_{1}^{k}.			
				\end{equation*}
			Recall $G_{0}$ is the orthonormal projection in $\CC^{n+1}$ to $V_{0}$ under the conical K\"ahler metric: $\II \p \pb R_{\tilde{v}}^{2}$. Pick $c$ sufficiently close to $0$ such that $2+c<p$
				\begin{equation*}
					|\nabla^{k}_{K^{-2}g_{\CC\times V_{0}}}\big(K^{-2}g_{\phi_{t}'}-G_{0}^{*}K^{-2}g_{\CC\times V_{0}}\big)|_{K^{-2}g_{\CC\times V_{0}}}<C_{k}(A_{1})\Lambda_{2}^{-2-c},
				\end{equation*}
			and 
				\begin{equation*}
					|\nabla^{k}_{K^{-2}g_{\phi_{t}}'}\big(K^{-2}g_{\phi_{t}'}-G_{0}^{*}K^{-2}g_{\CC\times V_{0}}\big)|_{K^{-2}g_{\phi_{t}}'}<C_{k}(A_{1})\Lambda_{2}^{-2-c}.
				\end{equation*}
			The error term in Ricci potential is
				{\begin{small}
				\begin{equation*}
						t^{2(q-p)}K^{2(1-p)}|Q_{t}'(v_{0}+Kv_{n+2})|^{2}\frac{(\II \p \pb |P(\tilde{v})|^{2/p}\varphi(P(\tilde{v})^{-1/p}\cdot \tilde{v})^{n+1}\wedge \II \ud v_{n+2}\wedge \ud \overline{v_{n+2}}}{\II \ud v_{1}\wedge \cdots \wedge \ud \overline{v}_{n+2}}.
				\end{equation*}
				\end{small}}
			Given that $t, \Lambda_{2}$ sufficiently large, the Ricci potential satisfies
				\begin{equation*}
					|\nabla^{k}_{K^{-2}g_{\phi_{t}}'}h_{t}|_{K^{-2}g_{\phi_{t}}'}<C_{k}(A_{1})t^{2(q-p)}K^{2(1-p)}.
				\end{equation*}
			To summarize
				\begin{equation*}
					\begin{split}
						|\nabla_{g_{\phi_{t}}'}^{k}(g_{\phi_{t}'}-g_{\CC\times V_{0}})|&<C_{k}(A_{1})\Lambda_{2}^{-2-c}R_{\tilde{z}}^{-k},\\
						|\nabla_{g_{\phi_{t}}'}^{k}h_{t}|&<C_{k}(A_{1})t^{^{2(q-p)}}R_{\tilde{z}}^{2(1-p)-k}.
					\end{split}
				\end{equation*}
			In particular if $2-\tau/2<p$
				\begin{equation*}
					R_{\tilde{z}}^{2-\tau+k}|\nabla_{g_{\phi_{t}}'}^{k}h_{t}|\lesssim t^{(4-\tau)(q-p)/p}.
				\end{equation*}
		\end{proof}		
	
		In the end, we now finish the proof of the main Proposition \ref{S_3_Prop_Ricci} of this section.
			\begin{proof}[Proof of Proposition \ref{S_3_Prop_Ricci}]
				In lemma \ref{S_3_Lemma_Ricci_1}, \ref{S_3_Lemma_Ricci_2}, \ref{S_3_Lemma_Ricci_3}, \ref{S_3_Lemma_RicciMet_4}, we establish the estimate for the Ricci potential in Region I-IV. It remains to show that for a certain choice of the parameter $\delta_{i}, \delta_{i}',\alpha,\dots$, there exits a positive constant $a$ such that
					\begin{equation*}
						\|h_{t}\|_{C^{0,\alpha}_{\bs{\delta-2},\tau-2,t}(X)}\lesssim t^{(2-\tau)(q-p)/p-a}
					\end{equation*}
				While any $\|f\|_{C^{0,\alpha}_{\bs{\delta-2},\tau-2,t}}\lesssim t^{(2-\tau)(q-p)/p-a}$, will also satisfies $\|f\|_{C^{0,\alpha}_{\bs{0},0,t}}\lesssim \epsilon(t)$. This is achieved by bounding the $C^{0,\alpha}_{\bs{0},0,t}$ norm by $C^{0,\alpha}_{\bs{\delta-2},\tau-2,t}$ norm on region I-IV respectively, and then solving a set of inequalities.
					\subsubsection*{Region I}
						Computation reads
							\begin{equation*}
								|f|\lesssim A_{1}^{\tau -\delta_{0}}\rho_{0}^{\delta_{0}-2}w_{0}^{\tau-2}t^{(2-\tau)(q-p)/p-a}<A_{1}^{\tau-\delta_{0}}\rho_{0}^{\delta_{0}-2q/p+\tau(q-p)/p}t^{-a}.
							\end{equation*}
						Similar computation holds for the H\"older semi-norm.
					\subsubsection*{Region II}
						We carry out the computation in the scaled model $(X_{i},t^{2\theta_{i}}g_{\phi_{t}'})$
							\begin{equation*}
								t^{\delta_{i}(\sigma_{i}'+\theta_{i})-\tau\sigma_{i}'-2\theta_{i}}|f|\lesssim (\rho_{i}')^{\delta_{i}-2}w_{i}^{\tau-2}t^{(2-\tau)(q-p)/p-a}.
							\end{equation*}
						And similarly for the H\"older semi-norm.
					
					\subsubsection*{Region III}
						In the original model, $|\xi_{a_{t}^{q-p}Q_{t}}|\gtrsim t^{(q-p+\sigma_{i}'q_{i})/p}$, the norm is
							\begin{equation*}
								|f|\lesssim\xi_{a_{t}^{q-p}Q_{t}}^{\tau-2}t^{(2-\tau)(q-p)/p-a}\lesssim t^{(\tau-2)\sigma'_{i}q_{i}/p-a},
							\end{equation*}
						and the semi-norm
							\begin{equation*}
								[f]_{0,\alpha}\lesssim t^{(\tau-2-\alpha)q_{i}\sigma_{i}'/p-\alpha(q-p)/p-a}.
							\end{equation*}
					\subsubsection*{Region IV}
						The computation goes
							\begin{equation*}
								|f|\lesssim R_{\tilde{z}}^{\tau-2}t^{(2-\tau)(q-p)/p-a}\lesssim t^{-a},
							\end{equation*}
						and the semi-norm
							\begin{equation*}
								[f]_{0,\alpha}\lesssim t^{-a+\alpha(p-q)/p}.
							\end{equation*}
					Now combine with the inequality \ref{S_3_Eqn_Ricci_2} we have the following inequalities for $a$ and $\sigma_{i}'$
						\begin{equation}\label{S_3_Ineqn_Parameter_1}
							\begin{cases}
								\delta_{i}'(\sigma_{i}'-\sigma_{i})+\delta_{i}'(\sigma_{i}+\theta_{i})-\tau\sigma_{i}'-2\theta_{1}<(2-\tau)(q-p)/p-a\\
								(3+\alpha-\tau)\sigma_{i}'+\alpha\theta_{i}-\tau\theta_{i}<(2-\tau)(q-p)/p-a\\
								\delta_{i}(\sigma_{i}'+\theta_{i})-\tau\sigma_{i}'-2\theta_{i}>(2-\tau)(q-p)/p-a\\
								(\tau-2-\alpha)q_{i}\sigma_{i}'/p-\alpha(q-p)/p-a<0\\
								-a+\alpha(p-q)/p<0
							\end{cases}.
						\end{equation}
					Without loss of generality, we consider the limiting situation when $\tau, \alpha\to 0, \sigma_{i}'\to \sigma_{i}$, and in this case the inequalities reduce to
						\begin{equation}\label{S_3_Ineqn_Parameter_2}
							\begin{cases}
								a<-\delta_{i}'(\sigma_{i}+\theta_{i})+2\theta_{i}+2(q-p)/p\\
								a<2(q-p)/p-3\sigma_{i}'\\
								a>2(q-p)/p+(2-\delta_{i})\theta_{i}-\delta_{i}\sigma_{i}\\
								a>-2q_{i}\sigma_{i}/p
							\end{cases}.
						\end{equation}
					Use Proposition \ref{S_2_Prop_Ricci_1}, we can set $\delta_{i}\to 2q_{i}/p$ and $\delta_{i}'\to 2(2q_{i}-p)/p$, the inequalities further reduce to
						\begin{equation}\label{S_3_Ineqn_Parameter_3}
							\begin{cases}
								a<2(p-q)/p-2\sigma_{i}(2q_{i}-p)/p\\
								a<2(q-p)/p-3\sigma_{i}\\
								a>-2q_{i}\sigma_{i}/p
							\end{cases}.
						\end{equation}
					The compatibility conditions for those inequalities and the constraint in our construction is
						\begin{equation*}
							\frac{q-p}{p}<\sigma_{i}<\frac{2(q-p)}{3p-2q_{i}}.
						\end{equation*}
					Here the condition $p>2q_{i}$ guarantees $2(q-p)/(3p-2q_{i})>(q-p)/p$. Define the interval
						\begin{equation*}
							I_{i}=(-\frac{2q_{i}\sigma_{i}}{p},\min\{\frac{2(q-p)}{p}-3\sigma_{i},\frac{2(p-q)}{p}-\frac{2\sigma_{i}(2q_{i}-p)}{p}\}).
						\end{equation*}
					Notice that when $i$ is fix, the left end point is monotonically increasing with respect to $-\sigma_{i}$. The infimum of the left end point is $\frac{4(p-q)q_{i}}{(3p-2q_{i})p}$, which is monotonically increasing with respect to $q_{i}$. On the other hand, the supremum of the right end point is
						\begin{equation*}
							\begin{cases}
								\frac{p-q}{p}, &q_{i}\geq\frac{1}{4}p, \sigma_{i}\to \frac{q-p}{p}\\
								\frac{2(p-q)(p+4q_{i})}{(5p-4q_{i})p}, & q_{i}<\frac{1}{4}p, \sigma_{i}\to \frac{4(q-p)}{5p-4q_{i}} 
							\end{cases},
						\end{equation*}
					which is again monotonically increasing with respect to $q_{i}$. Recall $q_{1}\leq \cdots\leq q_{d}$, and therefore to show that for a suitable choice of $\sigma_{i}$ the intersection $\bigcap_{i=1}^{d}I_{i}$ is not empty, it is suffice to for a certain choice of $\sigma_{i}$ the left end point of $I_{d}$ is less than the right end point of $I_{1}$. When $\frac{1}{4}p\leq q_{1}\leq q_{d}$ or $q_{1}\leq q_{d}<\frac{1}{4}p$, the above statement are automatically true. Meanwhile, when $q_{1}<\frac{1}{4}p\leq q_{d}$, we have the require in addition $\frac{p+4q_{1}}{5p-4q_{1}}>\frac{2q_{d}}{3p-2q_{d}}$. Now given that $\bigcap_{i=1}^{d}I_{i}$ is non-empty, for a particular choice of $a\in\bigcap_{i=1}^{d}I_{i}$, the inequalities \ref{S_3_Ineqn_Parameter_1} still hold after a small perturbation of the parameters.	
			\end{proof}
\section{A gluing theorem}\label{S_4}
	To sum up the discussion above, we have demonstrated that when $p,q_{i}$ satisfy a certain balancing condition \ref{S_3_Prop_Ricci}, the Ricci potential $h_{t}$ is small in both H\"older space and in the scaled models. In this section we use Banach fix point theorem to construct a K\"ahler metric $\omega_{\phi_{t}}$ such that its difference with our approximate solution $\omega_{\phi_{t}'}$
		\begin{equation*}
			\|\omega_{\phi_{t}}-\omega_{\phi_{t}'}\|_{C^{0,\alpha}_{\bs{\delta-2},\tau-2,t}}\lesssim t^{-(2-\tau)(p-q)/p-a}.
		\end{equation*}
	One should take special care of the region $U_{2,i}$, where the scaled metrics are approximated by $g_{\phi_{i}}$. We only have the weighted $C^{0,\alpha}$ estimate for the error in the gluing region (See \ref{S_3_Lemma_Met_2}, \ref{S_3_Remark_Met_2}). Consequently the difference between Christoffel symbols might not be bounded, however, on K\"ahler manifolds the difference between two scalar Laplacians is bounded by their $C^{0,\alpha}$ difference. Unlike the standard gluing theory, it does not make sense to define a right inverse for the Laplacian, for in the gluing region the H\"older space $C^{2,\alpha}_{\delta_{i}-2,\tau-2}(g_{\phi_{i}})$ and $C^{2,\alpha}_{\delta_{i}-2,\tau-2}(t^{2\theta_{i}}g_{\phi_{t}'})$ is not uniformly equivalent. We slightly modify the construction, namely we construct an approximate right inverse to a larger space
		\begin{equation*}
			P: C^{0,\alpha}_{\bs{\delta-2},\tau-2,t}\to C^{2,\alpha}_{\text{model},\bs{\delta},\tau,t},
		\end{equation*}
	where $C^{k,\alpha}_{\text{model},\bs{\delta},\tau,t}$ is the direct sum of the following spaces
		\begin{equation*}
			A_{1}^{\delta_{0}-\tau}\|\cdot\|_{C^{k,\alpha}_{\delta_{0},\tau}(U_{1}')}, t^{\delta_{i}(\sigma_{i}'+\theta_{i})-\tau\sigma_{i}'}\|\cdot \|_{C^{k,\alpha}_{\delta_{i},\tau}(U_{2,i}',g_{\phi_{i}})},\|\cdot \|_{C^{k,\alpha}_{\tau}(U_{3}')},\|\cdot \|_{C^{k,\alpha}_{\tau}(U_{4}')}.
		\end{equation*}
	Here, $U_{i,j}'$ contain $U_{i,j}$ and are neighborhoods of the support of the cutoff functions $\widetilde{\beta_{i,j}}$, which we shall define later in the next section. Moreover one can check norms in $C^{0,\alpha}_{\text{model},\bs{\delta},\tau,t}$ and $C^{0,\alpha}_{\bs{\delta},\tau,t}$ for tensors are uniformly equivalent, and similarly $C^{1,\alpha}_{\text{model},\bs{\delta},\tau,t}$, $C^{1,\alpha}_{\bs{\delta},\tau,t}$ for functions. Henceforth it makes sense to define an invertible operator
		\begin{equation*}
			\Delta_{\phi_{t}'}P: C^{0,\alpha}_{\bs{\delta-2},\tau-2,t}\to C^{0,\alpha}_{\bs{\delta-2},\tau-2,t}.
		\end{equation*}
	Define an operator from real valued function to real $(1,1)$ form
		\begin{equation}\label{S_4_Defn_R}
			\mathscr{R}:C^{0,\alpha}_{\bs{\delta-2},\tau-2,t}(\Lambda^{0})\to C^{0,\alpha}_{\bs{\delta-2},\tau-2,t}(\Lambda^{1,1}), f\mapsto \II \p \pb P(\Delta_{\phi_{t}'}P)^{-1}f.
		\end{equation}
	It remains to solve the following equation in the Banach space $C^{0,\alpha}_{\bs{\delta-2},\tau-2,t}(\Lambda^{0})$
		\begin{equation}\label{S_4_Eqn_CY}
			\frac{(\omega_{\phi_{t}'}+\mathscr{R}f)^{n+1}}{\II^{(n+1)^{2}}\Omega_{X_{t}}\wedge \overline{\Omega_{X_{t}}}}=1.
		\end{equation}
	The solution $\omega_{\phi_{t}'}+\mathscr{R}f$ is not smooth, but the regularity can be improved using Monge Amp\`ere equation. 
	\subsection{Cutoff functions}\label{S_4_1}
		The spirit of the cutoff functions is that $\beta_{i,j}$ is bounded in $C^{k,\alpha}_{\text{model},\bs{0},0,t}$, whereas $\widetilde{\beta_{i,j}}\equiv1$ on the support of $\beta_{i,j}$ and the weighted H\"older norms of $\nabla\widetilde{\beta_{i,j}}$ tend to $0$ as the gluing parameters $A_{1}, A_{2},\Lambda_{2}\to \infty$, which allow us to construct an approximate inverse for the Laplacian. And let $U_{i,j}'$ be a neighborhood of the support of $\widetilde{\beta_{i,j}}$. Define the cutoff functions explicitly
			\begin{equation*}
				\begin{split}
					&\beta_{1}=1-\gamma_{2}\big(\frac{R_{\tilde{z}}}{A_{1}}\big)\gamma_{2}\big(\frac{|z_{n+2}|}{A_{1}}\big),\\
					&\tilde{\beta_{1}}=1-\gamma_{2}\big(\frac{\ln(|z_{n+2}|A_{1}^{-1/2})}{\ln(A_{1}^{1/4})}\big)\gamma_{2}\big(\frac{\ln(R_{\tilde{z}}A_{1}^{-1/2})}{\ln(A_{1}^{1/4})}\big),\\	
					&\beta_{2,i}=\gamma_{2}\big(\frac{\rho_{i}}{t^{\sigma_{i}'}}\big), \widetilde{\beta_{2,i}}=\gamma_{2}\big(\frac{\ln(A_{2}\rho_{i}t^{-\sigma_{i}'}/2)}{\ln(A_{2})}\big),\\
					&\beta_{3}'=\gamma_{2}\big(\frac{|z_{n+2}|}{A_{1}}\big)\gamma_{2}\big(\frac{R_{\tilde{z}}}{t^{(q-p)/p}\Lambda_{2}}\big),\\ 
					&\widetilde{\beta_{3}'}=\gamma_{2}\big(\frac{\ln(R_{\tilde{z}}/2)}{\ln(A_{1})}\big)\gamma_{2}\big(\frac{\ln(R_{\tilde{z}}t^{(p-q)/p}/2)}{\ln(\Lambda_{2})}\big),\\
					&\beta_{3}=\beta_{3}'(1-\beta_{2,1}-\beta_{2,2}),\\
					 &\widetilde{\beta_{3}}=\widetilde{\beta_{3}'}\big(1-\gamma_{2}\big(\frac{\ln(\rho_{1}A_{2}^{-1/2}t^{-\sigma_{i}'})}{\ln(A_{2}^{1/4})}\big)-\gamma_{2}\big(\frac{\ln(\rho_{2}A_{2}^{-1/2}t^{-\sigma_{i}'})}{\ln(A_{2}^{1/4})}\big)\big),\\
					&\beta_{4}'=\gamma_{2}\big(\frac{R_{\tilde{z}}}{A_{1}}\big)\gamma_{2}\big(\frac{|z_{n+2}|}{A_{1}}\big)\gamma_{1}\big(\frac{R_{\tilde{z}}}{t^{(q-p)/p}\Lambda_{2}}\big),\\
					&\widetilde{\beta_{4}'}=\gamma_{2}\big(\frac{\ln(R_{\tilde{z}}/2)}{\ln(A_{1})}\big)\gamma_{2}\big(\frac{\ln(|z_{n+2}|/2)}{\ln(A_{1})}\big)\gamma_{1}\big(\frac{\ln(R_{\tilde{z}}\Lambda_{2}^{-1/2}t^{(p-q)/p})}{\ln(\Lambda_{2}^{1/4})}\big),\\
					&\beta_{4}=\beta_{4}'(1-\beta_{2,1}-\beta_{2,2}),\\
					&\widetilde{\beta_{4}}=\widetilde{\beta_{4}'}\big(1-\gamma_{2}\big(\frac{\ln(\rho_{1}A_{2}^{-1/2}t^{-\sigma_{i}'})}{\ln(A_{2}^{1/4})}\big)-\gamma_{2}\big(\frac{\ln(\rho_{2}A_{2}^{-1/2}t^{-\sigma_{i}'})}{\ln(A_{2}^{1/4})}\big)\big).
				\end{split}
			\end{equation*}

		On each model space we shall list the H\"older norms of each cutoff functions, and one should check the equivalence of the H\"older norm on the overlap of region I-IV. On region I, we see that
			\begin{equation*}
				|\nabla^{k} \beta_{1}|_{U_{1}'}< C_{k}\rho_{0}^{-k}, |\nabla^{k}\widetilde{\beta_{1}}|_{U_{1}'}<\frac{C_{k}}{\ln(A_{1})}\rho_{0}^{-k}.
			\end{equation*}
		Noting that $|w_{0}|\leq 1$, we now see that
			\begin{equation*}
				\|\beta_{1}\|_{C^{k,\alpha}_{0,0,t}(U_{1}')}<C_{k}, \|\nabla\widetilde{\beta_{1}}\|_{C^{k,\alpha}_{-1,-1,t}(U_{1}')}<\frac{C_{k}}{\ln(A_{1})}.
			\end{equation*}
		And similarly for the scaled model metrics near the critical points in region II, it should be pointed out that the H\"older space is defined using the model metrics but not the approximate solution
			\begin{equation*}
				\|\beta_{2,i}\|_{C^{k,\alpha}_{0,0}(U_{2,i}',g_{\phi_{i}})}<C_{k}, \|\nabla \widetilde{\beta_{2,i}}\|_{C^{k,\alpha}_{-1,-1}(U_{2,i}', g_{\phi_{i}})}<\frac{C_{k}}{\ln(A_{2})}.
			\end{equation*}
		On region III we can see that for the model metric $g_{3}=g_{\CC}\oplus t^{2(q-p)/p}|Q_{t}|^{2/d}g_{V_{1}}$. We note that
			\begin{equation*}
				\|\beta_{3}\|_{C^{k,\alpha}_{0}(g_{3})}<C_{k}, \|\nabla\widetilde{\beta_{3}}\|_{C^{k,\alpha}_{-1}(g_{3})}<C_{k}\max(\frac{1}{\ln A_{2}},\frac{1}{\ln{\Lambda_{2}}}).
			\end{equation*}
		On region IV we will arrive at the similar conclusion for the model metric $g_{4}=g_{\CC}\oplus g_{V_{0}}$
			\begin{equation*}
				\|\beta_{4}\|_{C^{k,\alpha}_{0}(U_{4}',g_{4})}<C_{k}, \|\nabla\widetilde{\beta_{4}}\|_{C^{k,\alpha}_{-1}(U_{4}', g_{4})}<C_{k}\max(\frac{1}{\ln A_{2}},\frac{1}{\ln{\Lambda_{2}}}).
			\end{equation*}
		Another ingredient is the multiplication law for our weighted H\"older space:
			\begin{lemma}\label{S_4_Lemma_Multiplication}
				Let $f_{1}\in C^{0,\alpha}_{\bs{\delta-2},\tau-2,t}$ and $f_{2}\in C^{0,\alpha}_{\bs{0},0,t}$, then there exists a constant $C$ independent on $t$ such that
					\begin{equation*}
						\|f_{1} f_{2}\|_{C^{0,\alpha}_{\bs{\delta-2},\tau-2,t}}\leq C\|f_{1}\|_{C^{0,\alpha}_{\bs{\delta-2},\tau-2,t}}\|f_{2}\|_{C^{0,\alpha}_{\bs{0},0,t}}.
					\end{equation*}
			\end{lemma}
			\begin{proof}
				Use the fact that $\|\cdot\|_{C^{0,\alpha}_{\bs{0},0,t}}$ is invariant under the scaling, and therefore we can transfer the computation to each scaled model spaces, and computation will show that the constant is independent on the $t$ parameter.
			\end{proof}

	\subsection{Decomposition and parametrix}\label{S_4_2}
		We mainly use Proposition \ref{S_2_Prop_Map_1}, \ref{S_2_Prop_Map_2}, \ref{S_2_Prop_Map_3}, \ref{MAP_LAPLACIAN} to construct the approximate right inverse, and we define the bounded right inverse in each model spaces as follow:
			\begin{definition}\label{S_4_Defn_Rightinverse}
				Define $P_{0}, P_{1}, P_{2,i}, P_{3}$ as the bounded right inverses to the Laplacians respectively
					\begin{equation*}
						\begin{split}
							&P_{0}: \Delta_{\CC\times V_{0}}, C^{2,\alpha}_{\delta_{0},\tau}(\CC\times V_{0})\to C^{0,\alpha}_{\delta_{0}-2,\tau-2}(\CC\times V_{0}),\\
							&P_{1}: \Delta_{\CC\times V_{1}},C^{2,\alpha}_{\tau}(\CC\times V_{1})\to C^{0,\alpha}_{\tau-2}(\CC\times V_{1}),\\
							&P_{2,i}: \Delta_{g_{\phi_{i}}}, C^{2,\alpha}_{\delta_{i},\tau}(X_{i})\to C^{0,\alpha}_{\delta_{i}-2,\tau-2}(X_{i}),\\
							&P_{3}: \Delta_{\CC\times V_{0}}, C^{2,\alpha}_{\tau}(\CC\times V_{0})\to C^{0,\alpha}_{\tau-2}(\CC\times V_{0}).
						\end{split}
					\end{equation*}
			\end{definition}
		The idea is to divide the function $f\in C^{0,\alpha}_{\bs{\delta-2},\tau-2,t}(X_{i})$ into several pieces
			\begin{equation*}
				f=f_{1}+f_{3}+f_{4}+\sum_{i=1}^{d}f_{2,i},
			\end{equation*}
		where $f_{i,j}=\beta_{i,j}f$. Our approach is to approximately invert $f_{i,j}$ in each regions, and then glue them together using the cutoff functions $\widetilde{\beta_{i,j}}$.
			\subsubsection*{Region I}
				Consider cutoff functions on region I
					\begin{equation*}
						\begin{split}
							&\beta_{1,i}=\gamma_{i}(R_{\tilde{z}}\Lambda_{1}^{-1}t^{(p-q)/p}\rho_{0}^{-q/p}),\\
							&\widetilde{\beta_{1,1}}=\gamma_{1}\bigg(\frac{\ln(R_{\tilde{z}}\Lambda_{1}^{-1/2}t^{(p-q)/p}\rho_{0}^{-q/p})}{\ln(\Lambda_{1}^{1/4})}\bigg),\\
							&\widetilde{\beta_{1,2}}=\gamma_{2}\bigg(\frac{\ln (R_{\tilde{z}}t^{(q-p)/p}2^{-1})}{\ln(\Lambda_{1})}\bigg).
						\end{split}
					\end{equation*}
				Further decompose $f_{1}$ into $f_{1}=f_{1,1}+f_{1,2}$, $f_{1,i}=\beta_{1,i}f$. In region I, we proceed in a similar manner as in \cite{szekelyhidi2019degenerations}. We use \ref{S_2_Prop_Map_1} to invert $f_{1,1}$ and \ref{S_2_Prop_Map_2} for $f_{2,2}$. The key difference of our construction to that of \cite{szekelyhidi2019degenerations} is that our cutoff functions and the weighted H\"older space on the $t$-parameter.
				
				Let $\chi_{i,t}\equiv 1$ on the ball $B_{v_{0}^{i}}(B_{i}|v_{0}^{i}|^{q/p}t^{(q-p)/p})$ in the base space and vanishes outside $B_{v_{0}^{i}}(2B_{i}|v_{0}^{i}|^{q/p}t^{(q-p)/p})$. And similarly define another cutoff fucntion, $\widetilde{X_{i,t}}\equiv 1$ on $B_{v_{0}^{i}}(2B_{1}|v_{0}^{i}|^{q/p}t^{(q-p)/p})$, and vanishes outside the $B_{v_{0}^{i}}(3B_{1}|v_{0}^{i}|^{q/p}t^{(q-p)/p})$. We shall now define the approximate inverse on region I
					\begin{equation}\label{S_4_Eqn_RI_1}
						P_{1}'(f)=P_{0}(\beta_{1,1}f_{1})+\sum_{j}\widetilde{\beta_{1,2}} \widetilde{\chi_{i,t}}|v_{0}^{i}|^{2q/p}t^{2(q-p)/p}P_{1}(\chi_{i,t}\beta_{1,2}f_{1}).
					\end{equation}
					\begin{lemma}\label{S_4_Lemma_Laplacian_1}
						We shall now estimate difference of the $\Delta_{g_{\phi_{t}'}}P_{1}'(f)$ and $f_{1}$ on $U_{1}'$
							\begin{equation*}
								\|\Delta_{g_{\phi_{t}'}}\widetilde{\beta_{1}}P_{1}'(f)-f_{1}\|_{C^{0,\alpha}_{\delta_{0},\tau}(U_{1}')}<\epsilon(\Lambda_{1},B_{1},A_{1}).
							\end{equation*}
					\end{lemma}
					\begin{proof}
						We shall use the approximation result for the metric in Lemma \ref{S_3_Lemma_Met_1}, in the region $U_{1,1}':=\{R_{\tilde{z}}>t^{(q-p)/p}\Lambda_{1}^{3/4}\rho_{0}^{q/p}\}$, the metric $g_{\phi_{t}'}$ is approximated by $g_{\CC\times V_{0}}$, and therefore
							\begin{equation*}
								\begin{split}
									&\Delta_{g_{\phi_{t}'}}\widetilde{\beta_{1,1}}P_{0}(f_{1,1})-f_{1,1}=(\Delta_{g_{\phi_{t}'}}\widetilde{\beta_{1,1}})P_{0}(f_{1,1})\\
									&+2\nabla\widetilde{\beta_{1,1}}\cdot \nabla P_{0}(f_{1,1})+(\Delta_{g_{\phi_{t}'}}-\Delta_{\CC\times V_{0}})f_{1,1}.
								\end{split}
							\end{equation*}
						Combine with estimate for $\widetilde{\beta_{1,1}}$
							\begin{equation*}
								\|\Delta_{g_{\phi_{t}'}}\widetilde{\beta_{1,1}}P_{0}(f_{1,1})-f_{1,1}\|_{C^{0,\alpha}_{\delta_{0}-2,\tau-2,t}(U_{1}')}<\epsilon(\Lambda_{1},A_{1})\|f\|_{C^{0,\alpha}_{\delta_{0}-2,\tau-2,t}(U_{1}')}.
							\end{equation*}
						Now on the region $U_{1,2,i}'$
							\begin{equation*}
								U_{1,2,i}':=\{|z_{n+2}-v_{0}^{i}|<3B_{1}t^{(q-p)/p}|v_{0}^{i}|^{q/p}, R_{\tilde{z}}<2t^{(q-p)/p}\Lambda_{1}^{2}\rho_{0}^{q/p}\}.
							\end{equation*}
						The metric is approximated by $t^{2(q-p)/p}|v_{0}^{i}|^{2q/p}g_{\CC\times V_{1}}$, and use lemma \ref{S_3_Lemma_Met_1} again
							\begin{equation*}
								\begin{split}
									&\|(\Delta_{g_{\phi_{t}'}}\widetilde{\beta_{1,2}}\widetilde{\chi_{i,t}}|v_{0}^{i}|^{2q/p}t^{2(q-p)/p}P_{1})(\chi_{i,t}f_{1,2})-\chi_{i,t}f_{1,2}\|_{C^{0,\alpha}_{\delta_{0},\tau}(U_{1,2,i}')}\\
									&<\epsilon(\Lambda_{1},B_{1},A_{1})\|f\|_{C^{0,\alpha}_{\delta_{0},\tau}(U_{1,2,i}')}.
								\end{split}
							\end{equation*}
						And use \ref{S_4_Lemma_Multiplication} we prove the lemma. 
					\end{proof}
			\subsubsection*{Region II}
				We can choose $U_{2,i}'=\{\rho_{i}<A_{2}^{2}t^{\sigma_{i}'}\}$ and in those region, we define
					\begin{equation}\label{S_4_Eqn_RI_2}
						P_{2,i}'(f)=t^{-2\theta_{i}}P_{2,i}(f_{2,i}).
					\end{equation}
				We then have the following lemma:
					\begin{lemma}\label{S_4_Lemma_Laplacian_2}
						If $t$ is sufficiently large, then the difference between $\Delta_{g_{\phi_{t}'}}P_{2,i}'f$ and $f$ can be bounded by
							\begin{equation*}
								\|\Delta_{g_{\phi_{t}'}}\widetilde{\beta_{2,i}}P_{2,i}'(f)-f_{2,i}\|_{C^{0,\alpha}_{\delta_{i}-2,\tau-2}(U_{2,i}',g_{\phi_{i}})}<\epsilon(A_{2})\|f\|_{C^{0,\alpha}_{\delta_{i}-2,\tau-2}(U_{2,i}',g_{\phi_{i}})}.
							\end{equation*}
					\end{lemma}
					\begin{proof}
						The proof follows from direct computation
							\begin{equation*}
								\begin{split}
									&\Delta_{g_{\phi_{t}'}}\widetilde{\beta_{2,i}}P_{2,i}'(f)-f_{2,i}=(\Delta_{g_{\phi_{t}'}}\widetilde{\beta_{2,i}})P_{2,i}'(f)\\
									&+2\nabla \widetilde{\beta_{2,i}}\cdot \nabla P_{2,i}'(f)+\widetilde{\beta_{2,i}}(\Delta_{g_{\phi_{t}'}}-\Delta_{t^{-2\theta_{i}}g_{\phi_{i}}})P_{2,i}'(f).
								\end{split}
							\end{equation*}
						Although in \ref{S_3_Lemma_Met_2} we only obtain a $C^{0,\alpha}_{0,0}$ bound for $t^{2\theta_{i}}g_{\phi_{t}'}-g_{\phi_{i}}$, the gluing map $F_{i}$ is biholomorphic. Henceforth the difference between those two K\"ahler metric is bounded by their $C^{0,\alpha}_{0,0}$ difference.
					\end{proof}
				\subsubsection*{Region III}
					We cover the region $U_{3}'=\{|z_{n+2}|<A_{1}^{2},\rho_{i}>A_{2}^{3/4}t^{\sigma_{i}'},R_{\tilde{z}}<t^{(q-p)/p}\Lambda_{2}^{2}\}$. Cover $U_{3}'$ by $N(t)\sim O(t^{(p-q)/p})$ many open sets $U_{3,i}'$
						\begin{equation*}
							U_{3,i}':=\{|z_{n+2}-v_{0}^{i,t}|<3B_{2}t^{(q-p)/p},R_{\tilde{z}}<t^{(q-p)/p}\Lambda_{2}^{2}\}.
						\end{equation*}
					We define the cutoff functions $\chi_{t,i}'$ such that $\chi_{t,i}'\equiv 1$ on $B_{v_{0}^{i,t}}(B_{2}t^{(q-p)/p})$ and vanish outside $B_{v_{0}^{i,t}}(2B_{2}t^{(q-p)/p})$, while $\widetilde{\chi_{t,i}'}\equiv 1$  on $B_{v_{0}^{i,t}}(2B_{2}t^{(q-p)/p})$ and vanish outside $B_{v_{0}^{i,t}}(3B_{2}t^{(q-p)/p})$. Now we define the approximate right inverse
						\begin{equation}\label{S_4_Eqn_RI_3}
							P_{3}'(f)=\widetilde{\beta_{3}}\sum_{i}\widetilde{\chi_{t,i}'}|Q_{t}(v_{0}^{i,t})|^{2/p}t^{2(q-p)/p}P_{1}(\chi_{t,i}'f_{3}).
						\end{equation}
					Similar use \ref{S_3_Lemma_Met_3}:
						\begin{lemma}\label{S_4_Lemma_Laplacian_3}
							When $t$ is sufficiently large
								\begin{equation}
									\|\Delta_{g_{\phi_{t}'}}P_{3}'(f)-f_{3}\|_{C^{0,\alpha}_{\tau-2,t}(U_{3}')}\lesssim(B_{2}^{-1}+\ln(A_{1})^{-1}+\ln(A_{2})^{-1})\|f\|_{C^{0,\alpha}_{\tau-2,t}(U_{3}')}.
								\end{equation}
						\end{lemma}
		\subsubsection*{Region IV}
			We choose $U_{4}'=\{|z_{n+2}|<A_{1}^{2},\rho_{i}>A_{2}^{3/4}t^{\sigma_{i}'},R_{\tilde{z}}>t^{(q-p)/p}\Lambda_{2}^{3/4}\}$, we define the right inverse
				\begin{equation}\label{S_4_Eqn_RI_4}
					P_{4}'=\widetilde{\beta_{4}}P_{3}(f_{4}).
				\end{equation}
			And we can estimate the difference:
				\begin{lemma}\label{S_4_Lemma_Laplacian_4}
					When $t$ is sufficiently large
						\begin{equation*}
							\|\Delta_{g_{\phi_{t}'}}P_{4}'(f)-f_{4}\|_{C^{0,\alpha}_{\tau}(U_{4}')}\lesssim (B_{2}^{-1}+\ln(A_{1})^{-1}+\ln(A_{2})^{-1}+\ln(\Lambda_{2})^{-1})\|f\|_{C^{0,\alpha}_{\tau}(U_{4}')}.
						\end{equation*}
				\end{lemma}
				\begin{proof}
					Use \ref{S_3_Lemma_RicciMet_4} we prove the lemma.
				\end{proof}
			
		We define $P$ as follow, and then justify that $P$ is the approximate right inverse for the scalar Laplacian.
			\begin{equation}\label{S_4_Eqn_RI}
				P=\widetilde{\beta_{1}}P_{1}'+\widetilde{\beta_{2}}P_{2}'+P_{3}'+P_{4}'.
			\end{equation}
		Although we are only able to control $P(f)$ in $C^{2,\alpha}_{\text{model}, \bs{\delta},\tau,t}$, but not in $C^{2,\alpha}_{\bs{\delta},\tau,t}$, we do know the behavior of $\Delta_{g_{\phi_{t}'}}P$ and $\II\p \pb P$. 
			\begin{prop}\label{S_4_Prop_RI}
				For any $\epsilon>0$, we first choose $\Lambda_{1}, A_{1}, B_{1}\to \infty$ and then choose $A_{2}, \Lambda_{2}$ large enough, and in the end choose $B_{2}$ sufficiently large, and then let the collapsing parameter $t$ goes to infinity, the the operator norm of $\Delta_{g_{\phi_{t}'}}P-Id$ on $C^{0,\alpha}_{\bs{\delta-2},\tau-2,t}$
					\begin{equation*}
						\|\Delta_{g_{\phi_{t}'}}P-Id\|<\epsilon.
					\end{equation*}
				And moreover $\II\p \pb P$ is a bounded operator on $C^{0,\alpha}_{\bs{\delta-2},\tau-2,t}$, with bound independent on $t$.
			\end{prop}
			\begin{proof}
				The first part follows directly from Lemma \ref{S_4_Eqn_RI_1}, \ref{S_4_Eqn_RI_2}, \ref{S_4_Eqn_RI_3}, \ref{S_4_Eqn_RI_4}. On the region $U_{1}', U_{3}', U_{4}'$ the second part is trivial since we have at least weighted $C^{2,\alpha}$ estimate for the metric. Meanwhile on $U_{2,i}'$, the gluing map $F_{i}$ are biholomorphism, which implies the $\II\p \pb$ operators are the same, and therefore we can use the fact that $C^{0,\alpha}_{\bs{\delta-2},\tau-2,t}$ and $C^{0,\alpha}_{\text{model},\bs{\delta-2},\tau-2,t}$ are uniformly equivalent to prove the proposition.
			\end{proof}
	\subsection{Collapsing Warped QAC Calabi-Yau metric}\label{S_4_3}
		In this section we carry out our gluing construction.
		
		\begin{theorem}\label{S_4_Thm_Gluing}
			Let $\alpha,\tau, \sigma_{i}'-\sigma_{i}$ sufficiently small. Assume $\delta_{i}\to \frac{2q_{i}}{p}, 1\leq i\leq d$ and $\bs{\delta}$ avoid the discrete sets of indicial roots as in Proposition \ref{MAP_LAPLACIAN}. Then the unique K\"ahler metric $\II \p \pb\phi_{t}$ solving the Calabi-Yau equation, approximated by $\II\p \pb\phi_{t}'$ near the infinity, is close to $\II \p \pb \phi_{t}'$, with the bound
				\begin{equation}
					\|\II \p \pb(\phi_{t}'-\phi_{t})\|_{C^{0,\alpha}_{\bs{\delta-2},\tau-2,t}}\lesssim t^{-(2-\tau)(p-q)/p-a},
				\end{equation}
			where $a$ is chosen in Proposition \ref{S_3_Prop_Ricci}, such that the difference is bounded by $\epsilon(t)$ on each scaled model spaces.
		\end{theorem}
		\begin{proof}
			Use the multiplication law in Lemma \ref{S_4_Lemma_Multiplication} and combine with Proposition \ref{S_3_Prop_Ricci}: if $\|f\|_{C^{0,\alpha}_{\bs{\delta-2},\tau-2,t}}\lesssim t^{(2-\tau)(q-p)/p-a}$ then $\|f\|_{C^{0,\alpha}_{\bs{0},0,t}}\lesssim \epsilon(t)$. We can therefore define a nonlinear operator.
				\begin{equation*}
					\mathscr{F}:C^{0,\alpha}_{\bs{\delta-2},\tau-2,t}\to C^{0,\alpha}_{\bs{\delta-2},\tau-2,t}, f\mapsto \frac{(\omega_{\phi_{t}'}+\mathscr{R}f)^{n+1}}{\II^{(n+1)^{2}}\Omega_{X_{t}}\wedge \overline{\Omega_{X_{t}}}}-1.
				\end{equation*}
			Recall that $\mathscr{R}f=\II \p \pb P(\Delta_{g_{\phi_{t}'}}P)^{-1}f$. We separate the linear and nonlinear parts of the operator $\mathscr{F}$
				\begin{equation*}
					\mathscr{F}(f)=\Tr(\mathscr{R}f)+\mathscr{N}(f)=f+\mathscr{N}(f).
				\end{equation*}
			We restrict our attention to a small open set:
				\begin{equation*}
					\mathscr{U}:=\{f\in C^{0,\alpha}_{\bs{\delta-2},\tau-2,t},\|f\|<t^{-(2-\tau)(p-q)/p-a}\}.
				\end{equation*}
			For $f_{1},f_{2}\in \mathscr{U}$
				\begin{equation*}
					\|\mathscr{N}(f_{1})-\mathscr{N}(f_{2})\|<C(\|f_{1}\|+\|f_{2}\|)\|f_{1}-f_{2}\|\lesssim \epsilon(t)\|f_{1}-f_{2}\|,
				\end{equation*}
			when $t$ is sufficiently large, $\mathscr{N}$ map $\mathscr{U}$ to itself and we can apply Banach fix point theorem to obtain a unique solution $f=-\mathscr{N}(f), f\in \mathscr{U}$ with estimate as required.
		\end{proof}
		\begin{corollary}\label{S_4_Cor_pGH}
			Our gluing construction provides explicit families of non-collapse complete Calabi-Yau manifolds $(X_{t}, \II \p \pb \phi_{t}, J_{t}, pt_{t})$ that converges in the pointed-Gromov-Hausdorff sense (See Section 2.2 \cite{donaldson2017gromov}) to $(\CC\times V_{0}, \omega_{\CC}+\omega_{V_{0}}, J, pt)$, for a suitable choice of base points $pt_{t}, pt$.
		\end{corollary}
		\begin{proof}
			We shall use the hypersurface $X_{t}$ given by
				\begin{equation*}
					P(\tilde{z})+a_{t}^{q-p}Q_{t}(z_{n+2})=0.
				\end{equation*}
			Recall that $G_{0}$ is the geodesy projection from $\CC^{n+1}$ to $V_{0}$ under the singular K\"ahler metric $\II \p \pb R^{2}$, we define a map
				\begin{equation*}
					F_{0}: \CC^{n+1}\times \CC\to V_{0}\times C, F_{0}=G_{0}\times Id.
				\end{equation*}
			The restriction of $F_{0}$ on $X_{t}$ defines a diffeomorphism $F_{t}$ from and open set in $X_{t}$
				\begin{equation*}
					U_{t}:=\{(\tilde{z},z_{n+2})\in X_{t}|R_{\tilde{z}}>\max{(A /\log(t), \Lambda |z_{n+2}|^{\frac{q}{p}}/\log(t))}\},
				\end{equation*}
			to an open set in $\CC\times V_{0}$. Pick a base point $pt$ in the smooth part of $\CC\times V_{0}$ and define $pt_{t}$ as $F_{t}^{-1}(pt)$. Let $U_{t,D}(pt_{t})$ be the intersection of $U_{t}$ with a geodesic ball of radius $D$, which centers at $pt_{t}$. It follows from the Theorem \ref{S_4_Thm_Gluing} that, $\forall \epsilon>0$, $U_{t,D}$ is $\epsilon$-dense in the ball and the image of $U_{t,D}$ is $\epsilon$-close to $B(pt, D)$ in $\CC\times V_{0}$. Moreover we can compare the metrics and the complex structure
				\begin{equation*}
					\|F_{t}^{*}g_{\CC\times V_{0}}-g_{\phi_{t}}\|_{C^{k}(g_{\phi_{t}})}+\|F_{t}^{*}J_{\CC\times V_{0}}-J\|_{C^{k}(g_{\phi_{t}})}\leq C(k,D)\epsilon(t).
				\end{equation*}
			Therefore we prove the corollary.
		\end{proof}
		
\section{Bubble tree}\label{S_5}
	In the section we will discuss a bubble tree structure for the collapsing warped QAC Calabi-Yau manifolds construction in Theorem \ref{Thm_2} in a slightly more general setting. For simplicity, we assume the degree for the Calabi-Yau cone $V_{0}$ and any degree of the polynomials involved with respect to $z_{n+2}$ satisfy the numerical constraints in Theorem \ref{Thm_1}.\\
	
	We will begin with constructing a tree structure on the set of $q$ distinct zeros. We label each of zero and denote the set of zeros as $Z$. Moreover let $\mathscr{I}$ be a subset of the power set $\mathscr{P}(Z)$, such that:
		\begin{enumerate}
			\item $\emptyset \notin \mathscr{I}, Z\in \mathscr{I} $.
			\item If $I_{1}, I_{2}\in \mathscr{I}$ and $I_{1}\cap I_{2}\neq \emptyset$, then either $I_{1}\subseteq I_{2}$ or $I_{2} \subseteq I_{1}$.
			\item If $I_{1}, I_{2}\in \mathscr{I}$, $I_{1}\subsetneq I_{2}$ and there doesn't exist $I_{3}\in \mathscr{I}$ such that $I_{1}\subsetneq I_{3}\subsetneq I_{2}$, then there exists $I_{j} \in \mathscr{I}$ such that $I_{2}\setminus I_{1}=\sqcup I_{j}$.
		\end{enumerate}
	Now we define the depth of $I$ to be the maximum of $n$ such that
		\begin{equation}\label{S_5_Eqn_Depth}
			I=I_{n}\subsetneq \cdots \subsetneq I_{0}=Z.
		\end{equation}
	Moreover the depth of $\mathscr{I}$ is defined to be $\max_{I \in \mathscr{I}}(\depp{I})$. An element $I\in \mathscr{I}$ is call minimal if there doesn't exist an $I'\in \mathscr{I}$ such that $I'\subsetneq I$. It is easy to see that $Z=\sqcup_{I \text{ minimal}}I$.\\
	
	\begin{definition}\label{S_3_Defn_Tree}
		A tree structure on $(\mathscr{I}, \subsetneq)$ is defined in the following way
		\begin{enumerate}
			\item $\mathscr{I}$ is the set of vertex.
			\item There is an edge between $I_{1}, I_{2}\subset \mathscr{I}$ if and only if $I_{1} \subsetneq I_{2}$ or $I_{2}\subsetneq I_{1}$ and $|\depp{I_{1}}-\depp{I_{2}}|=1$.
		\end{enumerate}
	\end{definition}
	\begin{remark}\label{S_5_Remark_Valent}
		It follows from the third property for $\mathscr{I}$ that if $I$ is not minimal and $\depp{I}>0$, then there are at least $3$ edges connecting to $I$.
	\end{remark}
	\begin{definition}\label{S_3_Defn_Root_Subtree}
		We will now introduce several key concepts and notation for future argument.
		\begin{enumerate}
			\item \textbf{Root}: we call the unique element of depth $0$ a root. For example, $Z$ is the root of $(\mathscr{I},\subsetneq)$.
			\item \textbf{Subtree}: Let $K\in \mathscr{I}$ there is a induced tree structure on
				\begin{equation*}
					(\{I\subset K| I \in \mathscr{I}\}, \subsetneq).
				\end{equation*}
				We call this tree, the subtree induced by $K$, and $K$ is the root of this subtree.
			\item \textbf{Forest}: Forest is a disjoint union of trees. For example, let $\{K_{j}\}$ be the set of element of depth $k$, the union of subtrees induced by $K_{j}$ is a forest.
			\item \textbf{Path}: For each $K_{0}\subsetneq K_{n}$ we define
				\begin{equation*}
					K_{0}\subsetneq K_{1}\subsetneq \cdots \subsetneq K_{n}, \depp{K_{j+1}}=\depp{K_{j}}+1,
				\end{equation*}
				to be the path from $K_{0}$ to $K_{n}$.
			\item \textbf{Set of labelled points}: To each $I$ we define the set of labeled points $L(I)$: If $I$ is minimal, $L(I):=I$; If $I$ is not minimal then $L(I):=\{I_{j}\}$, where $I=\sqcup I_{j}$ with $\depp{I_{j}}=\depp{I}+1$. We assign $|L(I)|$ distinct points on $\CC$ and label them by $L(I)$
					\begin{equation*}
						b_{\alpha}\in \CC, \alpha\in L(I).
					\end{equation*}
			\item $q_{I}$: We denote $q_{I}$ as the number of zeros contained in $I\in \mathscr{I}$.
			\item \textbf{Colliding parameters}: For every $I\in \mathscr{I}$, we define a positive real parameter $t_{I}$, such that $t_{I_{1}}\gg t_{I_{2}}$ if $I_{2}\subsetneq I_{1}$. We will use $\bs{t}=(t_{I}), I \in \mathscr{I}$ to denote those parameters. Roughly speaking, $t_{I}$ measures the rate of colliding for elements in $L(I)$. For simplicity, we will let $t_{Z}=1$.
		\end{enumerate}
	\end{definition}
	
	Now we give a parametrization of an element in $Z$. Let $\alpha\in L(I)$, where $I$ is minimal and $\depp{I}=n\geq 1$. Let $I\subsetneq K$ and
		\begin{equation*}
			K=I_{k}\subsetneq \cdots \subsetneq I_{n}=I
		\end{equation*}
	be the path from $K$ to $I$. We define
		\begin{equation*}
			s_{\bs{t},K,\alpha}:=t_{K}^{-1}\bigg(t_{I}b_{\alpha}+\sum_{j=k}^{n-1} t_{I_{j}} b_{I_{j+1}}\bigg).
		\end{equation*}
	When $K=Z$ is the root we will simply write $s_{\bs{t},K,\alpha}$ as $s_{\bs{t},\alpha}$. We define a polynomial with zeros labeled by $Z$
		\begin{equation*}
				Q_{\bs{t}}(z_{n+2}):=\prod_{ I \text{ minimal}}\prod_{\alpha\in L(I)}(z_{n+2}-s_{\bs{t},\alpha}).
		\end{equation*}
		
	Similarly to the discussion in Section \ref{S_3}, we define the $K$-blow up model for polynomial $Q_{\bs{t}}$
		\begin{equation*}
			Q_{\bs{t},K}(z_{_{n+2}})=C_{K}\prod_{I \text{ minimal}, I\subset K} \prod_{\alpha\in L(I)}(z_{n+2}-s_{\bs{t},K,\alpha}).
		\end{equation*}
	Here $C_{K}$ is a constant, which is defined inductively:
		\begin{enumerate}
			\item $C_{Z}=1$,
			\item If $K\subsetneq K'$ and $\depp{K}=\depp{K'}+1$ and $K'\setminus K =\sqcup_{j=1}^{k} K_{j}$, then
				\begin{equation*}
					C_{K}=C_{K'}\prod_{j=1}^{k}(b_{K}-b_{K_{j}})^{q_{K_{j}}}, K_{j}\in L(K').
				\end{equation*}
		\end{enumerate}
		
	Moreover define a hyper-surface and its $K$-blow up models in $\CC^{n+2}$ accordingly
		\begin{equation*}
			\begin{split}
				&X_{\bs{t}}: P(\tilde{z})+Q_{\bs{t}}(z_{n+2})=0,\\
				&X_{\bs{t},K}:P(\tilde{z})+Q_{\bs{t},K}(z_{n+2})=0.
			\end{split}
		\end{equation*}
	Given a minimal $I$ and $K\in\mathscr{I}, I \subset K$, we also define marked point in $X_{\bs{t},K}$ as follow
		\begin{equation*}
			pt_{\bs{t}, K,\alpha}=(0,s_{\bs{t},K,\alpha}),\alpha \in L(I).
		\end{equation*}
	It can be checked directly that when $K$ is minimal $(X_{\bs{t},K}, pt_{\bs{t},K,\alpha})$ is independent on $\bs{t}$. In this case we will simply write it as $(X_{K}, pt_{K,\alpha})$.\\
	 
	We apply Theorem \ref{Thm_1} to obtain a Calabi-Yau K\"ahler form $\II \p \pb \phi_{\bs{t},K, u_{K}}$ on $X_{\bs{t},K}$ approximated by
		\begin{equation*}
			\II \p \pb \phi_{\bs{t},K, u_{K}}'(z)=\II \p \pb \bigg(|z_{n+2}|^{2}+u_{K}^{-2}|P(\tilde{z})|^{2/p}\varphi(P(\tilde{z})^{-1/p}\cdot \tilde{z})\bigg), u_{K}>0.
		\end{equation*}
	Here $u_{K}$ is a collapsing parameter. Our gluing theorem \ref{Thm_2} corresponds to the case when the depth of the tree is $1$. Moreover let $I$ be a minimal element, the parameter
			\begin{equation*}
				t_{Z}=1,t_{I}=t^{-\frac{p-q}{p-q_{I}}}, u_{Z}=t^{\frac{p-q}{p}}, u_{I}=1.
			\end{equation*}
		In this case, when $t$ is large enough, the gluing theorem implies that the Calabi-Yau metric $\II t_{I}^{-2}\p \pb \phi_{\bs{t},Z,u_{Z}}$ on a small neighborhood of $pt_{Z,\alpha}$, $\alpha\in L(I)$ can be approximated by $\II \p \pb \phi_{I,u_{I}}$ on $X_{I}$. Our task for the next subsection is to use this gluing theorem as a stepping stone to give a preferable description for the bubble tree convergence.
	
\subsection{Convergence for bubble tree}\label{S_5_1}
	We shall apply a similar heuristic argument in Section \ref{S_3}. Let $K_{2}\subsetneq K_{1}$, $\depp{K_{2}}=\depp{K_{1}}+1$. After the coordinate change
		\begin{equation*}
			((t_{K}^{-1}t_{I})^{q_{I}/p}\cdot \tilde{v},t_{K}^{-1}t_{I}v_{n+2}+b_{\alpha_{I}})=(\tilde{z},z_{n+2}),
		\end{equation*}
	$(X_{\bs{t},K_{1}},\II \p \pb \phi_{\bs{t},K_{1},u_{K_{1}}}')$ is approximated by
		\begin{equation*}
			(X_{\bs{t},K_{2}}, (t_{K_{1}}^{-1}t_{K_{2}})^{2}\II \p \pb \phi_{\bs{t},K_{2},u_{K_{2}}}),
		\end{equation*}
	where
		\begin{equation*}
			u_{K_{2}}=u_{K_{1}}(t_{K_{1}}^{-1}t_{K_{2}})^{(p-q_{K_{2}})/p}.
		\end{equation*}
	We require that when $I$ is minimal $u_{I}$ is a fixed constant. 
		\begin{lemma}\label{S_5_Lemma_compatible}
			For every choice of $\{u_{I}\}$, where $I$ is minimal, there exists choices of $\{t_{K'}\}, K'\in \mathscr{I}$, such that for every non-minimal $K\in \mathscr{I}$ of depth $k$, the collapsing parameter $u_{K}=u_{k}$ only depends on its depth, and that for any minimal $I$ contained in $K$, $u_{K}$ satisfies
				\begin{equation*}
					u_{K}=u_{I}\prod_{j=k+1}^{n}(t_{I_{j}}^{-1}t_{I_{j-1}})^{(p-q_{I_{j}})/p}.
				\end{equation*}
			Here $I=I_{n}\subsetneq \cdots \subsetneq I_{k}=K$ is the path from $K$ to $I$. Moreover, $t_{I_{j}}^{-1}t_{I_{j-1}}$ can be arbitrary large.
		\end{lemma}
		\begin{proof}
			For simplicity, let
				\begin{equation*}
					x_{I_{j},I_{j-1}}=(t_{I_{j}}^{-1}t_{I_{j-1}})^{(p-q_{I_{j}})/p}.
				\end{equation*}
			There is an one-to-one correspondence between $x_{I_{j},I_{j-1}}$ and the set of edges of $(\mathscr{I},\subsetneq)$. We prove this lemma by induction. Suppose the depth of $\mathscr{I}$ is $n$. We begin with the forest consisting of the subtrees induced by non-minimal $K\in \mathscr{I}, \depp{K}=n-1$, the conclusion follows directly, by letting $x_{I,K}=u_{n-1}u_{I}^{-1}$ for some arbitrary large $u_{n-1}$. Now suppose the statement is true for $\depp{K'}\geq k+1$, consider the forest induced by non-minimal elements of depth $k$. Pick an arbitrary root $K$ of a tree in the forest. We shall choose $u_{k}$ sufficiently large so that $\forall K'\in L(K)$
				\begin{equation*}
					x_{K',K}=u_{k}u_{K'}^{-1},
				\end{equation*}
				is sufficiently large. Therefore the statement is also true for any tree of depth $n$.
		\end{proof}
		
		\begin{proposition}\label{S_5_Prop_Bubble_Convergence}
			Given a tree $(\mathscr{I},\subsetneq)$ with $Z$ being its root, such that
				\begin{enumerate}
					\item $q_{Z}<p$, where $p$ is the degree of the Calabi-Yau cone.
					\item For any non-minimal $K \in \mathscr{I}$, let $K_{j} \in L(K),1\leq j\leq d$, $\{q_{K_{j}}\}$ satisfies the numerical constraint in Theorem \ref{Thm_2}.
					\item For every minimal $I \in \mathscr{I}$, a collapsing parameter $u_{I}$.
				\end{enumerate}
			Then there exists a sequence of warped QAC Calabi-Yau manifolds
				\begin{equation*}
					(X_{\bs{t}_{l},Z}, \II \p \pb \phi_{Z, u_{Z,l}}).
				\end{equation*}
			For every minimal $I$, we associated it with a sequence of marked points
				\begin{equation*}
					pt_{\bs{t}_{l},\alpha}=(0,s_{\bs{t}_{l},\alpha})\in X_{\bs{t}_{l},Z}, \alpha\in L(I),
				\end{equation*}
			such that
				\begin{equation*}
					(X_{\bs{t}_{l},Z}, pt_{\bs{t}_{l},\alpha}, \II t_{I,l}^{-2}\p \pb \phi_{Z, u_{Z,l}},J)
				\end{equation*}
			is pointed-Gromov-Hausdorff convergent to
				\begin{equation*}
					(X_{I}, pt_{\alpha}, \II \p \pb \phi_{I, u_{I}},J),
				\end{equation*}
			where $pt_{\alpha}=(0,s_{\alpha}), \alpha\in L(I)$.
		\end{proposition}
		\begin{proof}
			Suppose the depth of $\mathscr{I}$ is $n$, let $K$ be an arbitrary non-minimal element of depth $n-1$. For every $I\in L(K)$, we choose
				\begin{equation*}
					pt_{\bs{t}_{l}, K,\alpha}=(0,s_{\bs{t}_{l},K,\alpha}),\alpha\in L(I)
				\end{equation*}
			and $u_{n-1}=u_{K}$ large enough such that
				\begin{equation}\label{S_5_Eqn_Estimate}
					\|(t_{I}^{-1}t_{K})^{2}(F^{-1}_{I,K,\bs{t}_{l}})^{*}\big(\II \p \pb \phi_{\bs{t}_{l},K,u_{K,l}}\big)-\II \p \pb \phi_{I,u_{I}}\|_{C^{0,\alpha}_{0,0}(B_{l}, g_{\phi_{I},u_{I}})}<\frac{1}{n}l^{-\epsilon}.
				\end{equation}
			Here $F_{I,K,\bs{t}_{l}}$ is a biholomorphism from a large open set in $X_{I}$ to a small neighborhood of $pt_{\bs{t}_{l},\alpha,K}$ in $X_{\bs{t}_{l},K}$, which is similar to the one constructed in Section \ref{S_3}. The ball $B_{l}$ centers at $pt_{\bs{t}_{l},\alpha,K}$ and its radius is $\log(l)$. We can iterate this process inductively, now let $K$ be an arbitary non-minimal element of depth $k$, we choose $u_{k}=u_{K}$ large enough so that for every minimal $I\subsetneq K$
				\begin{enumerate}
					\item If $\depp{I}=k+1$, then the same estimate in Equation \ref{S_5_Eqn_Estimate} holds.
					\item If $\depp{I}>k+1$, let $I\subsetneq K'\subsetneq K$ with $\depp{K'}=\depp{K}+1$, and the pull-back scaled Calabi-Yau metric near $pt_{\bs{t}_{l},\alpha,K}$ is sufficiently closed to $(X_{\bs{t}_{l},K'}, \II \p \pb \phi_{\bs{t}_{l},K',u_{K',l}})$ on a very large open set, such that
						\begin{equation*}
							\|(t_{I}^{-1}t_{K,l})^{2}(F^{-1}_{I,K,\bs{t}_{l}})^{*}\big(\II \p \pb \phi_{\bs{t}_{l},K,u_{K,l}}\big)-\II \p \pb \phi_{I,u_{I}}\|_{C^{0,\alpha}_{0,0}(B_{l}, g_{\phi_{I},u_{I,}})}<\frac{n-k}{n}l^{-\epsilon}.
						\end{equation*}
				\end{enumerate}
			Use induction, we conclude that when $l$ is large enough, for each minimal $I$, there exists a biholomorphism $F_{I,\bs{t}_{l}}$ from a small neighborhood of $pt_{I,l}$ in $X_{\bs{t}_{l},Z}$ to a ball $B(pt_{\alpha},\log{l})$ in $X_{I}$. And the corresponding Calabi-Yau K\"ahler forms satisfies the estimate
				\begin{equation*}
					\|t_{I}^{-2}(F^{-1}_{I,\bs{t}_{l}})^{*}\big(\II \p \pb \phi_{Z,u_{Z,l}}\big)-\II \p \pb \phi_{I,u_{I}}\|_{C^{0,\alpha}_{0,0}(B_{l}, g_{\phi_{I},u_{I}})}<l^{-\epsilon}.
				\end{equation*}
			It should be remarked that we let $t_{Z}=1$ for simplicity. We can again use Monge-Amp\`ere equation to improve the regularity. It follows from definition that this pointed sequence of Calabi-Yau manifolds is pointed-Gromov-Hausdorff convergent to $(X_{I},pt_{\alpha}, \II \p \pb \phi_{I,u_{I}},J)$ in the sense of \cite{donaldson2017gromov}.
		\end{proof}

\section*{Acknowledgement}	
	The author is grateful to his Ph.D. advisor Claude LeBrun and Simon Donaldson for their suggestions, discussion and encouragement. He would also like to thanks Charlie Cifarelli, Jiaji Cai, Ronan Conlon, Yichen Cheng, Hanbing Fang, Shuo Gao, Siqi He, Yang Li, Junbang Liu, Fr\'ed\'eric Rochon, Jian Wang and Junsheng Zhang for fruitful discussion and helpful comments.\\
	
	This paper is partially supported by Simons Foundation International, LTD. Most of the work was completed during the author's stay in Morningside Center. 

\nocite{*}

\bibliographystyle{alpha}

\end{document}